\documentclass[11pt]{amsart}
\usepackage[dvipsnames,usenames]{color}
\usepackage[colorlinks=true, urlcolor=NavyBlue, linkcolor=NavyBlue, citecolor=NavyBlue]{hyperref}
\usepackage{graphicx}
\usepackage{epsfig}
\usepackage[latin1]{inputenc}
\usepackage{amsmath}
\usepackage{amsfonts}
\usepackage{amssymb}
\usepackage{amsthm}
\usepackage{amscd}
\usepackage{verbatim}
\usepackage{subfigure}
\usepackage{pinlabel}
\usepackage{stmaryrd}
\usepackage{enumerate, enumitem}

\usepackage{tikz}
\usetikzlibrary{arrows}
\usetikzlibrary{decorations.pathreplacing}
\usepackage{verbatim}
	
\newtheorem{theorem}{Theorem}[section]
\newtheorem{Theorem}{Theorem}
\newtheorem{lemma}[theorem]{Lemma}
\newtheorem{proposition}[theorem]{Proposition}

\newtheorem{Corollary}[Theorem]{Corollary}

\theoremstyle{definition}
\newtheorem{definition}[theorem]{Definition}

\theoremstyle{remark}
\newtheorem{remark}[theorem]{Remark}

\newtheorem*{ack}{Acknowledgements}
\newtheorem*{org}{Organization}

\def\Z{\mathbb{Z}}
\def\N{\mathbb{N}}
\def\R{\mathbb{R}}

\def\F{\mathbb{F}}

\def\cA{\mathcal{A}}

\def\cC{\mathcal{C}}
\def\cO{\mathcal{O}}

\def\bfx{\mathbf{x}}
\def\bfy{\mathbf{y}}
\def\CFKi{CFK^{\infty}}
\def\horz{\textup{horz}}
\def\vert{\textup{vert}}
\def\cCFK{\mathcal{CFK}}
\def\cCFKa{\mathcal{CFK}_\textup{alg}}
\def\varep{\varepsilon}
\def\min{\textup{min}}
\def\even{\textup{even}}
\def\odd{\textup{odd}}
\def\ba{\boldsymbol{a}}

\def\top{\textup{top}}
\def\varep{\varepsilon}

\title{An infinite rank summand of topologically slice knots}

\subjclass[2013]{}
\author[Jennifer Hom]{Jennifer Hom}
\thanks{The author was partially supported by NSF grant DMS-1307879.}
\address {Department of Mathematics, Columbia University, 2990 Broadway \\ New York, NY 10027}
\email{hom@math.columbia.edu}

\numberwithin{equation}{section}

\begin{document}

\begin{abstract}
Let $\cC_{TS}$ be the subgroup of the smooth knot concordance group generated by topologically slice knots. Endo showed that $\cC_{TS}$ contains an infinite rank subgroup, and Livingston and Manolescu-Owens showed that $\cC_{TS}$ contains a $\Z^3$ summand. We show that in fact $\cC_{TS}$ contains a $\Z^\infty$ summand. The proof relies on the knot Floer homology package of Ozsv\'ath-Szab\'o and the concordance invariant $\varep$.
\end{abstract}

\maketitle

\section{Introduction}
The \emph{knot concordance group}, denoted $\cC$, consists of knots in $S^3$ modulo smooth concordance, with the operation induced by connected sum. The identity is the equivalence class of \emph{slice knots}, that is, knots with are concordant to the unknot. The inverse of a knot $K$ is given by $-K$, the reverse of the mirror image of $K$. In 1969, Levine \cite{Levine1, Levine2} constructed a surjective homomorphism
\[ \cC \rightarrow \Z^\infty \oplus \Z_2^\infty \oplus \Z_4^\infty, \]
defined in terms of certain classical algebraic invariants. In particular, $\cC$ contains a summand isomorphic to $\Z^\infty$. Casson-Gordon \cite{CassonGordon2} showed that the kernel of Levine's homormophism has non-trivial kernel. However, in higher dimensions, the corresponding homomorphism is a bijection, and thus completely determines the structure of higher dimensional knot concordance groups.

If one only requires the concordance to be locally flat, then we obtain the \emph{topological concordance group}, $\cC_\top$. There is a homomorphism
\[ \cC \rightarrow \cC_\top \]
obtained by forgetting the smooth structure on the concordance.
We denote the kernel of the homomorphism $ \cC \rightarrow \cC_\top$ by $\cC_{TS}$, that is, the subgroup of the smooth knot concordance group generated by topologically slice knots. One difficulty in studying $\cC_{TS}$ is that
Levine's homomorphism factors through the topological concordance group, and thus cannot be used to study $\cC_{TS}$. Furthermore, many other of the most revealing tools in knot concordance (e.g., Casson-Gordon invariants \cite{CassonGordon2, CassonGordon}, the Cochran-Orr-Teichner filtration \cite{COT}) also vanish on $\cC_{TS}$.

Applying Donaldson's \cite{Donaldsongauge} work on the intersection forms of $4$-manifolds, Akbulut and Casson observed (see \cite{CochranGompf}) that there exist knots with trivial Alexander polynomial that are not smoothly slice. Combining this result with Freedman's  work \cite{Freedman}, which shows that any knot with trivial Alexander polynomial is topologically slice, implies that the kernel $\cC_{TS}$ is non-trivial.
In 1995, Endo \cite{Endo}, using gauge theory, showed that there is a $\Z^{\infty}$ subgroup of $\cC_{TS}$ and the fact that the subgroup $\cC_{TS}$ is highly non-trivial illustrates that there is indeed a large difference between the smooth and topological categories. In particular, the non-triviality of $\cC_{TS}$ implies the existence of exotic $\R^4$s, that is, smooth manifolds homeomorphic, but not diffeomorphic, to $\R^4$ (see, for example, \cite{GompfStipsicz}). 

A main goal in the study of knot concordance is to better understand the structure of $\cC_{TS}$. While we know that $\cC_{TS}$ contains infinitely generated subgroups \cite{Endo, HeddenKimLiv}, remarkably little is known about splitting and divisibility in $\cC_{TS}$. The first result in this direction is due to Livingston \cite{Livingstoncomp}, who, using the Ozsv\'ath-Szab\'o $\tau$ invariant \cite{OS4ball}, showed that $\cC_{TS}$ contains a summand isomorphic to $\Z$ and thus contains elements that are not divisible. 

Building upon techniques of Ozsv\'ath-Szab\'o \cite{OSabsgr}, Khovanov \cite{Khovanov}, and Rasmussen \cite{Rs}, Livingston's $\Z$-summand result can be improved to a $\Z^3$-summand \cite{Livingstondistinct, ManolescuOwens}, using the smooth concordance homomorphisms given by the Rasmussen $s$ invariant, and Manolescu-Owens $\delta$ invariant. In particular, they showed that $\tau$, $s$, and $\delta$ are linearly independent, even when restricted to topologically slice knots.
Further improvements in this direction have been severely limited by the scarcity of tractable $\Z$-valued smooth concordance homomorphisms.

In this paper, we use the $\{-1, 0, 1\}$-valued concordance invariant $\varepsilon$ associated to the knot Floer complex to prove the following.

\begin{Theorem}
\label{thm:main}
The group $\cC_{TS}$ contains a summand isomorphic to $\Z^\infty $.
\end{Theorem}

\noindent The knot Floer complex, defined by Ozsv\'ath-Szab\'o \cite{OSknots} and independently Rasmussen \cite{R}, associates to a knot $K$ in $S^3$ a $\Z \oplus \Z$-filtered chain complex $\CFKi(K)$. The concordance invariant $\varepsilon$ is defined in terms of certain natural maps on subquotient complexes of $\CFKi(K)$ \cite{Homcables}.

In \cite{Homsmooth}, we showed that $\varepsilon$, while itself \emph{not} a concordance homomorphism, can nonetheless be used to construct a highly nontrivial homomorphism from $\cC$ to a non-Archimedean totally ordered group. We will use this construction to define infinitely many linearly independent $\Z$-valued smooth concordance homomorphisms, which remain linearly independent when restricted to the topologically slice subgroup.
The results of this paper demonstrate the extent to which properties of this ordering vastly broaden our access into the structure of the concordance group.

There is significant work devoted to understanding the kernel $\cA$ of Levine's homomorphism
\[ \cC \rightarrow \Z^\infty \oplus \Z_2^\infty \oplus \Z_4^\infty. \]
Indeed, for nearly a decade from when the homomorphism was defined in 1969 until the advent of Casson-Gordon invariants \cite{CassonGordon2, CassonGordon}, it was unknown that $\cA$ was non-trivial. Further work (e.g., \cite{Jiang}, \cite{COT2}, \cite{CochranTeichner}) revealed deeper structure in $\cA$. However, these techniques were unable to determine anything about splitting in $\cA$. As in the case of $\cC_{TS}$, the strongest splitting result \cite{ManolescuOwens, Livingstondistinct} relies on the invariants $\tau$, $s$, and $\delta$, and splits off a $\Z^3$ summand. We have the following corollary, suggested by Livingston.

\begin{Corollary}
\label{cor:main}
The kernel of Levine's homomorphism contains a summand isomorphic to $\Z^\infty$.
\end{Corollary}

The filtered chain homotopy type of $\CFKi(K)$ is an invariant of knot $K \subset S^3$. The set of such complexes, modulo the relation \emph{$\varepsilon$-equivalence}
\[ C_1 \sim_\varepsilon C_2 \quad \textup{ if and only if } \quad \varepsilon(C_1 \otimes C_2^*)=0, \]
forms a group, denoted $\cCFK$, with the operation induced by tensor product. (Here, $C^*$ denotes the dual of $C$.) Moreover, since
\[ \CFKi(K_1 \# K_2) \simeq \CFKi(K_1) \otimes \CFKi(K_2) \quad \textup{ and } \quad \CFKi(-K) \simeq \CFKi(K)^*, \]
it follows that there exists a group homomorphism 
\[ \cC \rightarrow \cCFK \]
defined by sending $[K]$ to $[\CFKi(K)]$. Furthermore, the group $\cCFK$ is totally ordered \cite[Proposition 4.1]{Homsmooth}, with the ordering given by
\[ C_1 > C_2 \quad \textup{ if and only if } \quad \varepsilon(C_1 \otimes C_2^*) =1.\]
This structure of a total ordering is quite rich; see, for example, \cite{HancockHomNewman} and \cite{Tobin}. In this paper, we use this structure to define a family of concordance homomorphisms.

Given a totally ordered abelian group $G$, the \emph{absolute value} of an element $g \in G$ is defined to be
\[ |g|=\left\{
	\begin{array}{ll}
		g & \text{if } g \geq \textup{id}_G \\
		-g & \text{otherwise}.
	\end{array} \right.\]
\begin{definition}
Two non-zero elements $g$ and $h$ of a totally ordered group are \emph{Archimedean equivalent}, denoted $g \sim_A h$, if there exists a natural number $N$ such that $N \cdot |g| > |h|$ and $N \cdot |h| > |g|$. If $g$ and $h$ are not Archimedean equivalent and $|g|<|h|$, we say that $h$ \emph{dominates} $g$ and we write $|g| \ll |h|$. 
\end{definition}

\noindent It is straightforward to verify that $h \sim_A h+g$ for any $|g| \ll |h|$.

The set of Archimedean equivalence classes is naturally totally ordered as follows: $[g] < [h]$ if $|g| \ll |h|$. 
The Hahn Embedding Theorem \cite{Hahn} states that a totally ordered abelian group $G$ embeds in an order preserving way into $\R^X$, where
\begin{itemize}
	\item $X$ denotes the ordered set of Archimedean equivalence classes of $G$ 
	\item $\R^X$ consists of elements $(a_x) \in \prod_{x \in X} \R$ such that the set of $x$ for which $a_x$ is non-zero has a greatest element
	\item $\R^X$ is ordered with the reverse lexicographical ordering.
\end{itemize}
Thus, we may construct a homomorphism from $\cC$ to $\R$ as the composition
\begin{align}
\label{eqn:homomorphism}
 \cC \rightarrow \cCFK \hookrightarrow \R^X  \overset{\pi}\rightarrow \R,
\end{align}
where the first map sends $[K]$ to $[\CFKi(K)]$, the second map is given by the Hahn Embedding Theorem, and the third is projection onto a factor of $\R^X$. Our goal is to show that for an infinite linearly independent family of such homomorphisms, the image is isomorphic to $\Z$.

Given $(a_x) \in \R^X$ and $y \in X$, let $(a_x)_{> y}$ be the truncation of $(a_x)$ at $y$. More precisely, $(a_x)_{> y}$ is the element $(b_x) \in \R^X$, where
\[ b_x=\left\{
	\begin{array}{ll}
		a_x & \text{if } x > y \\
		0 & \text{otherwise}.
	\end{array} \right.\]	
By \cite[Theorem 3.1]{HausnerWendel} together with \cite{Clifford}, after choosing a representative\footnote{In general, $X$ may be infinite and we may have to appeal to the axiom of choice. However, for the group $\cCFK$, a selection rule might be available, for example if each Archimedean equivalence class of $\cCFK$ has a unique representative $\CFKi(K)$ of minimal rank.} $g_x \in G$ for each $x \in X$, the embedding
\[ \varphi: G \rightarrow \R^X \]
may be chosen such that:
\begin{enumerate}
	\item \label{it:char} $\varphi(g_x) \in \R^X$ is the element consisting of all zeros except in the coordinate corresponding to $x \in X$, where it has a value of one.
	\item \label{it:trun} If $(a_x) \in \varphi(G)$, then so is $(a_x)_{> y}$ for any $y \in X$.
\end{enumerate}

\noindent Note that the homomorphism $\varphi: G \rightarrow \R^X$ will depend on the choice of representative for each Archimedean equivalence class  of $G$.

\begin{definition}[Property A]
Let $g \in G$. The element $g$ \emph{satisifies Property A} 
if for every $h \in G$ such that $g \sim_A h$, we have that $h = k \cdot g + \mathcal{O}_A(g)$, where $k$ is an integer and $\mathcal{O}_A(g)$ consists of elements that are dominated by $g$.
\end{definition}

\noindent The following proposition relates Property A to homomorphisms to the integers.

\begin{proposition}
\label{prop:propertyA}
For each $g \in G$ that satisfies Property A, there exists a surjective homomorphism from $G$ to $\Z$. Furthermore, given $g_1, \ldots, g_n$ such that each $g_i$ satisfies Property A and $0< g_1 \ll \dots \ll g_n$, this collection of homomorphisms from $G$ to $\Z$ is linearly independent.
\end{proposition}

Thus, in order to prove Theorem \ref{thm:main}, it is sufficient to find an infinite family of topologically slice knots $K_n$, $n \in \N$, such that
\begin{itemize}
	\item For each $n$, we have that $0< [\CFKi(K_n)] \ll [\CFKi(K_{n+1})]$.
	\item Each $[\CFKi(K_n)] \in \cCFK$ satisfies Property A.
\end{itemize}
Our knots $K_n$ will be
\[ K_n = D_{n, n+1} \# -T_{n, n+1} \]
where $D_{n, n+1}$ denotes the $(n, n+1)$-cable of the (untwisted, positively clasped) Whitehead double of the right-handed trefoil and $T_{n, n+1}$ denotes the $(n, n+1)$-torus knot. (Our convention for cabling is that the first parameter denotes the longitudinal winding.)
Such knots were used in \cite[Remark 4.9]{Homsmooth}, where the author showed that $0< [\CFKi(K_n) \ll [\CFKi(K_{n+1})]$ for each $n$. Thus, in order to complete the proof of Theorem \ref{thm:main}, we must show that each $[\CFKi(K_n)] \in \cCFK$ satisfies Property A. The body of the paper is devoted to this goal.


We conclude the introduction with the proof of Proposition \ref{prop:propertyA}.

\begin{proof}[Proof of Proposition \ref{prop:propertyA}]
Let $x \in X$ denote the Archimedean equivalence class of $g$, and 
\[ \varphi: G \rightarrow \R^X \] 
be an embedding satisfying (\ref{it:char}) and (\ref{it:trun}) above, where $g$ is the chosen representative of its Archimedean equivalence class. Let 
\[ \pi_x: \R^X \rightarrow \R \]
denote the projection from $\R^X$ to the coordinate corresponding to $x$. In particular, note that $\pi_x \circ \varphi (g)=1$. 
We will show that for all $h \in G$,
\[ \pi_x \circ \varphi (h) \in \Z. \]

We first consider the case where $|h| \ll |g|$. Since $\varphi$ is an order-preserving homomorphism, it follows that $|\varphi(h)| \ll |\varphi(g)|$. Thus, the fact that $\pi_x \circ \varphi (g)=1$ implies that $\pi_x \circ \varphi(h)$ must equal zero. 

Next, we consider the case where $g$ and $h$ are Archimedean equivalent. By hypothesis, the element $g$ satisfies Property A, and so $h=k\cdot g + \cO (g)$. Then
\begin{align*}
	\pi_x \circ \varphi (h) &= \pi_x \circ \varphi \big(k\cdot g + \cO (g)\big) \\
	&= k \cdot \pi_x \circ \varphi (g) \\
	&= k
\end{align*}
since $ \pi_x \circ \varphi (g)=1$ and $\varphi \big(\cO (g)\big) =0$. 

Finally, we consider the case where $|h| \gg |g|$. By (\ref{it:trun}) above, we have that $\varphi(h)_{>x}$ is in the image of $\varphi$. It follows that $\varphi(h)-\varphi(h)_{>x}$ is also in the image of $\varphi$, i.e., $\varphi(h')=\varphi(h)-\varphi(h)_{>x}$ for some $h'\in G$. Moreover, the element $\varphi(h')$ consists of zeros in each coordinate $y$ where $y>x$. Thus, $\varphi(h')$ is either Archimedean equivalent to $\varphi(g)$ or dominated by $\varphi(g)$, and so $h'$ is either Archimedean equivalent to $g$ or dominated by $g$. In particular, by the considerations above, $\pi_x \circ \varphi(h')$ is an integer. Moreover, we have that
\begin{align*}
	\pi_x \circ \varphi(h') &= \pi_x \big(\varphi(h)-\varphi(h)_{>x}\big) \\
	&= \pi_x \circ \varphi(h) - \pi_x \circ \varphi(h)_{>x}\\
	&= \pi_x \circ \varphi(h)
\end{align*}
since $\pi_x \circ \varphi(h)_{>x}=0$. Thus, $\pi_x \circ \varphi(h)$ is an integer as well. 

To prove the second statement in the lemma, observe that if $x_i \in X$ is the Archimedean equivalence class of $g_i$, then
\[ \pi_{x_i} \circ \varphi(g_j) =\left\{
	\begin{array}{ll}
		1 & \text{if } i=j \\
		0 & \text{otherwise},
	\end{array} \right.\]	
since $\varphi(g_j)$ is the element consisting of all zeros except in the coordinate corresponding to $x_j$.
This completes the proof of the proposition.
\end{proof}

\begin{org}
The paper is organized as follows.
In Section \ref{sec:HFK}, we recall important facts about the knot Floer complex $\CFKi$, the invariant $\varepsilon$, and the group $\cCFK$. Next, in Section \ref{sec:numinvts}, we define certain numerical invariants of the $\varepsilon$-equivalence class of $\CFKi$. In Section \ref{sec:comp}, we relate these numerical invariants to Property A, and in Section \ref{sec:knots}, we compute these invariants for the knots $K_n$. Finally, in Section \ref{sec:proof}, we complete the proof of Theorem \ref{thm:main}.
\end{org}

\begin{ack}
I would like to thank the Simons Center for its hospitality during Spring 2013 when part of this work was done, and Liam Watson for a helpful conversation. I would also like to thank Peter Horn, Charles Livingston, and Peter Ozsv\'ath for comments on earlier drafts of this paper.
\end{ack}

\section{Knot Floer homology and the invariant $\varepsilon$}
\label{sec:HFK}
We assume that the reader is familiar with knot Floer homology, defined by Ozsv\'ath-Szab\'o \cite{OSknots} and independently Rasmussen \cite{R}. We briefly recall some important properties of their construction below.

\subsection{The knot Floer complex}
To a knot $K \subset S^3$, Ozsv\'ath-Szab\'o and Rasmussen associate a $\Z$-filtered, free, finitely generated chain complex over $\F[U, U^{-1}]$, where $\F=\Z/2\Z$ and $U$ is a formal variable. This complex is denoted by $\CFKi(K)$, and the filtration is called the Alexander filtration. The filtered chain homotopy type of $\CFKi(K)$ is an invariant of the knot $K$. There is a second $\Z$-filtration on $\CFKi(K)$ induced by $-(U\textup{-exponent})$, giving $\CFKi(K)$ the structure of a $\Z \oplus \Z$-filtered chain complex. We call this second filtration the $w$-filtration, since the filtration comes from the $w$ basepoint on the Heegaard diagram. We denote the two filtrations by the pair $(w(x), A(x))$. The (partial) ordering on $\Z \oplus \Z$ is given by $(i, j) \leq (i', j')$ if $i \leq i'$ and $j \leq j'$. Multiplication by $U$ shifts each filtration down by one; that is, if $x$ has filtration level $(i, j)$, then $U \cdot x$ has filtration level $(i-1, j-1)$. The complex obtained by reversing the roles of $i$ and $j$ is filtered chain homotopic to the original complex. The knot Floer complex is endowed with a homological, or Maslov, grading. The differential lowers the Maslov grading by one, and multiplication by $U$ lowers the Maslov grading by two. Up to filtered chain homotopy, we may assume that the complex $\CFKi(K)$ is reduced; that is, that the differential strictly lowers the filtration.

The knot Floer complex is well-behaved under connected sum \cite[Theorem 7.1]{OSknots} and taking mirrors \cite[Section 3.5]{OSknots}. We have that
\[ \CFKi(K_1 \# K_2) \simeq \CFKi(K_1) \otimes_{\F[U, U^{-1}]} \CFKi(K_2) \]
and
\[ \CFKi(-K) \simeq \CFKi(K)^*, \]
where $\CFKi(K)^*$ denotes the dual of $\CFKi(K)$, i.e., $\textup{Hom}_{\F[U, U^{-1}]}(\CFKi(K), \F[U, U^{-1}])$. The knot Floer complex is not sensitive to changes in the orientation of $K$.

A \emph{filtered basis} for $\CFKi(K)$ is a basis over $\F[U, U^{-1}]$ that induces a basis for the associated graded object. Given a filtered basis, it is often convenient to visualize the complex $\CFKi(K)$ in the $(i, j)$-plane. That is, we plot a generator $x$ at the lattice point corresponding to its filtration level $(w(x), A(x))$. We depict the differential using arrows; if $y$ appears in $\partial x$, then we place an arrow from $x$ to $y$. Note that the differential always points non-strictly down and to the left, because $\CFKi(K)$ is a doubly filtered chain complex.  
See Figure \ref{fig:CFKiex} for examples. Note that the depiction of $\CFKi(K)^*$ can be obtained from that of $\CFKi(K)$ by reversing the direction of all of the arrows as well as the direction of both filtrations.

Given a filtered basis, we will often perform a change of basis to obtain a new basis. If this new basis is also filtered, then we call the change of basis a \emph{filtered change of basis}. The most common type of filtered change of basis in this paper will be replacing a basis element $x$ with the sum of $x$ and elements of lower or equal filtration. This sum will have the same filtration level as $x$.

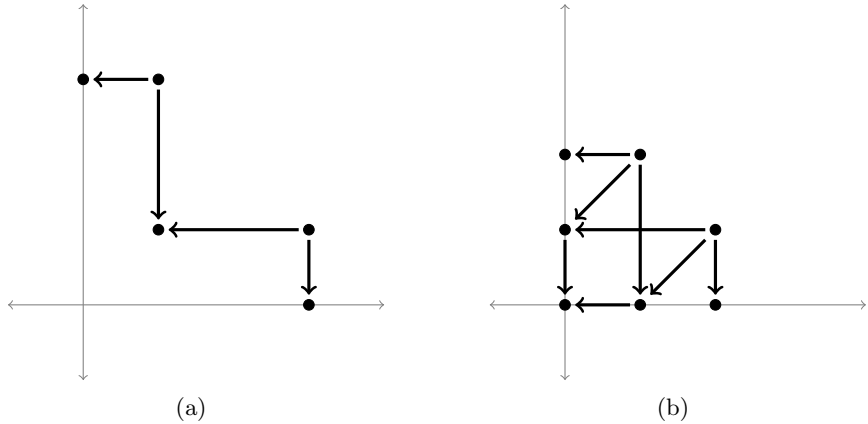
\begin{figure}[htb!]
\vspace{5pt}
\labellist
\small \hair 2pt
\endlabellist
\centering
\subfigure[]{
\begin{tikzpicture}
	\begin{scope}[thin, gray]
		\draw [<->] (-1, 0) -- (4, 0);
		\draw [<->] (0, -1) -- (0, 4);
	\end{scope}
	\filldraw (0, 3) circle (2pt) node[] (a){};
	\filldraw (1, 3) circle (2pt) node[] (b){};
	\filldraw (1, 1) circle (2pt) node[] (c){};
	\filldraw (3, 1) circle (2pt) node[] (d){};
	\filldraw (3, 0) circle (2pt) node[] (e){};
	\draw [very thick, <-] (a) -- (b);
	\draw [very thick, <-] (c) -- (b);
	\draw [very thick, <-] (c) -- (d);
	\draw [very thick, <-] (e) -- (d);
\end{tikzpicture}
\label{subfig:T_3,4}
}
\hspace{25pt}
\subfigure[]{
\begin{tikzpicture}
	\begin{scope}[thin, gray]
		\draw [<->] (-1, 0) -- (4, 0);
		\draw [<->] (0, -1) -- (0, 4);
	\end{scope}
	\filldraw (0, 2) circle (2pt) node[] (a){};
	\filldraw (1, 2) circle (2pt) node[] (b){};
	\filldraw (1, 0) circle (2pt) node[] (c){};
	\filldraw (0, 0) circle (2pt) node[] (d){};
	\filldraw (0, 1) circle (2pt) node[] (e){};
	\filldraw (2, 1) circle (2pt) node[] (f){};
	\filldraw (2, 0) circle (2pt) node[] (g){};
	\draw [very thick, <-] (a) -- (b);
	\draw [very thick, <-] (c) -- (b);
	\draw [very thick, <-] (d) -- (c);
	\draw [very thick, <-] (d) -- (e);
	\draw [very thick, <-] (e) -- (f);
	\draw [very thick, <-] (g) -- (f);
	\draw [very thick, <-] (e) -- (b);
	\draw [very thick, <-] (c) -- (f);
\end{tikzpicture}
}
\caption{Left, $\CFKi(T_{3,4})$. Right, $\CFKi((T_{2,3})_{2, 1})$.}
\label{fig:CFKiex}
\end{figure}

At times, we will want to consider only the part of the differential that preserves the Alexander filtration, i.e., the horizontal arrows, which we denote $\partial ^\horz$. Similarly, we may want to consider only the part of the differential that preserves the $w$-filtration, i.e., the vertical arrows, which we denote $\partial^\vert$.

We will often consider various subquotient complexes of $\CFKi(K)$. Let $S$ be a subset of $\Z \oplus \Z$. We write $C\{S\}$ to denote the set of elements whose $(i, j)$-coordinates are in $S$, together with the arrows between them. If $S$ has the property that $(i, j) \in S$ implies that $(i', j') \in S$ for all $(i', j') \leq (i, j)$, then $C\{S\}$ is a subcomplex of $\CFKi(K)$. For example, $CFK^-(K)$ is the subcomplex $C\{ i \leq 0\}$, i.e., the left half place. The complex $\widehat{CFK}(K)$ is the subquotient complex $C\{ i=0\}$; that is, it is the quotient of the subcomplex $C\{i \leq 0\}$ by $C\{i<0\}$.

\subsection{Concordance invariants}
We now define the concordance invariants $\tau$ and $\varepsilon$, both associated to the knot Floer complex and defined in terms of the (non-)vanishing of certain natural maps on homology.

The total homology of $\widehat{CFK}(K)$ is isomorphic to $\widehat{HF}(S^3) \cong \F$. The Ozsv\'ath-Szab\'o integer-valued concordance invariant $\tau(K)$ measures the minimal filtration level that supports the homology; more precisely,
\[ \tau(K) = \textup{min} \{ s \mid \iota:C\{i=0, j \leq s\} \rightarrow C\{i=0\} \textup{ induces a non-trivial map on homology}\}. \] 
The invariant $\tau$ gives a surjective homomorphism from $\cC$ to the integers. Moreover, this homomorphism is non-trivial on $\cC_{TS}$, giving a $\Z$ summand of $\cC_{TS}$ \cite{Livingstoncomp}.

The $\{-1, 0, 1\}$-valued concordance invariant $\varepsilon$ \cite{Homsmooth} is similarly defined in terms of the vanishing of certain natural maps on subquotient complexes. Let $\tau=\tau(K)$. Consider the map
\[ F: C\{i=0\} \rightarrow C\{ \textup{min}(i, j-\tau) \} \]
consisting of quotienting by $C\{ i=0, j<\tau\}$ followed by inclusion. The induced map on homology corresponds to the $2$-handle cobordism from $S^3$ to $S_{-N}^3(K)$ in a particular spin$^c$ structure for sufficiently large $N$ \cite[Section 4]{OSknots}. Similarly, we may consider the map
\[ G: C\{ \textup{max}(i, j-\tau) \} \rightarrow C\{i=0\} \]
consisting of quotienting by $C\{i <0, j=\tau\}$, followed by inclusion. Here, the induced map on homology corresponds to the $2$-handle cobordism from $S^3_N$ to $S^3$ in a particular spin$^c$ structure.

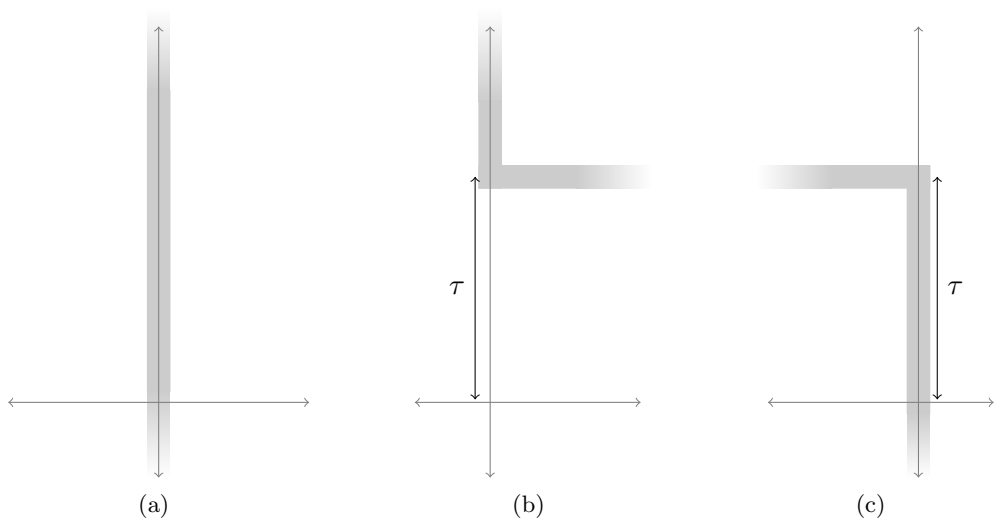
\begin{figure}[htb!]
\vspace{5pt}
\labellist
\small \hair 2pt
\endlabellist
\centering
\subfigure[]{
\begin{tikzpicture}
	\filldraw[black!20!white] (-0.15, 0.15) rectangle (0.15, 4.15);
	\shade[top color=white, bottom color=black!20!white] (-0.15, 4.15) rectangle (0.15, 5.25);
	\shade[top color=white, top color=black!20!white] (-0.15, -1) rectangle (0.15, 0.15);
	\begin{scope}[thin, gray]
		\draw [<->] (-2, 0) -- (2, 0);
		\draw [<->] (0, -1) -- (0, 5);
	\end{scope}
\end{tikzpicture}
}
\hspace{25pt}
\subfigure[]{
\begin{tikzpicture}
	\filldraw[black!20!white] (-0.15, 2.85) rectangle (0.15, 4);
	\filldraw[black!20!white] (-0.15, 2.85) rectangle (1.15, 3.15);
	\shade[top color=white, bottom color=black!20!white] (-0.15, 4) rectangle (0.15, 5.25);
	\shade[top color=white, left color=black!20!white] (1.15, 2.85) rectangle (2.15, 3.15);
	\begin{scope}[thin, gray]
		\draw [<->] (-1, 0) -- (2, 0);
		\draw [<->] (0, -1) -- (0, 5);
	\end{scope}
	\draw [<->] (-0.2, 0.04) -- (-0.2, 3) node [midway, left] {$\tau$};
\end{tikzpicture}
}
\hspace{25pt}
\subfigure[]{
\begin{tikzpicture}
	\filldraw[black!20!white] (-0.15, -0.15) rectangle (0.15, 3.15);
	\filldraw[black!20!white] (-1.15, 2.85) rectangle (0.15, 3.15);
	\shade[top color=white, top color=black!20!white] (-0.15, -1) rectangle (0.15, -0.15);
	\shade[top color=white, right color=black!20!white] (-2.15, 2.85) rectangle (-1.15, 3.15);
	\begin{scope}[thin, gray]
		\draw [<->] (-2, 0) -- (1, 0);
		\draw [<->] (0, -1) -- (0, 5);
	\end{scope}
	\draw [<->] (0.25, 0.04) -- (0.25, 3) node [midway, right] {$\tau$};
\end{tikzpicture}
}

\caption{Left, the subquotient complex $C \{ i=0 \}$. Center, the subquotient complex $C\{ \textup{min} (i, j-\tau)=0\}$. Right, the subquotient complex $C\{ \textup{max } (i, j-\tau)=0\}$.}
\end{figure}

The concordance invariant $\varepsilon(K)$ is defined \cite[Definition 3.1]{Homsmooth} to be
\begin{itemize}
	\item $\varepsilon(K)=-1$ if $G_*$ is trivial (in which case $F_*$ is necessarily non-trivial)
	\item $\varepsilon(K)=0$ if both $F_*$ and $G_*$ are non-trivial
	\item $\varepsilon(K)=1$ if $F_*$ is trivial (in which case $G_*$ is necessarily non-trivial).
\end{itemize}
The invariant $\varepsilon$ can also be defined in terms of a basis for $\CFKi$ and certain arrows between basis elements. A \emph{vertically simplified basis} is a basis $\{x_i\}$ over $\F[U, U^{-1}]$ for $\CFKi$ such that for each basis element $x_i$ exactly one of the following is true:
\begin{itemize}
	\item There is a unique outgoing vertical arrow from $x_i$ and no incoming vertical arrows.
	\item There is a unique incoming vertical arrow to $x_i$ and no outgoing vertical arrows.
	\item There are no incoming or outgoing vertical arrows to $x_i$.
\end{itemize}
We may similarly define a horizontally simplified basis. Since the homology of any column of $\CFKi$ has rank one, given a vertically simplified basis, there is a unique basis element with no incoming or outgoing vertical arrows; we call such an element the \emph{vertically distinguished element}. By \cite[Lemmas 3.2 and 3.3]{Homcables}, there exists a horizontally simplified basis such that there is a unique basis element $x_0$ that is the \emph{vertically} distinguished element of some vertically simplified basis. Then the invariant $\varepsilon$ can be defined to be
\begin{itemize}
	\item $\varepsilon(K)=-1$ if $x_0$ has a unique outgoing horizontal arrow (and no incoming horizontal arrows)
	\item $\varepsilon(K)=0$ if $x_0$ has no in or outgoing horizontal arrows
	\item $\varepsilon(K)=1$ if $x_0$ has a unique incoming horizontal arrow (and no outgoing horizontal arrows).
\end{itemize}

\begin{remark}
The examples in Figure \ref{fig:CFKiex} have bases that are simultaneously vertically and horizontally simplified. However, not every $\Z \oplus \Z$-filtered chain complex admits a simultaneously vertically and horizontally simplified basis. See Figure \ref{fig:nobasis} for an example of such a chain complex. It remains open whether or not every $\Z \oplus \Z$-filtered chain complex that arises as $\CFKi(K)$ for some knot $K \subset S^3$ admits a simultaneously vertically and horizontally simplified basis.
\end{remark}

\begin{figure}[htb!]
\vspace{5pt}
\labellist
\small \hair 2pt
\endlabellist
\centering
\begin{tikzpicture}
	\filldraw (-0.1, 2.1) circle (2pt) node[] (a){};
	\filldraw (0.1, 1.9) circle (2pt) node[] (b){};
	\filldraw (2.1, 2.1) circle (2pt) node[] (c){};
	\filldraw (1.9, 1.9) circle (2pt) node[] (d){};
	\filldraw (1, 1) circle (2pt) node[] (e){};
	\filldraw (-0.1, -0.1) circle (2pt) node[] (f){};
	\filldraw (0.1, 0.1) circle (2pt) node[] (g){};
	\filldraw (1.9, 0.1) circle (2pt) node[] (h){};
	\filldraw (2.1, -0.1) circle (2pt) node[] (i){};
	\draw [very thick, <-] (a) -- (c);
	\draw [very thick, <-] (b) -- (d);
	\draw [very thick, <-] (i) -- (c);
	\draw [very thick, <-] (h) -- (d);
	\draw [very thick, <-] (f).. controls (-0.4, 0.4) and (-0.4, 1.6) .. (a);
	\draw [very thick, <-] (g) -- (b);
	\draw [very thick, <-] (f) .. controls (0.4, -0.45) and (1.6, -0.45) .. (i);
	\draw [very thick, <-] (g).. controls (0.4, 0.2) and (1.6, 0.2) .. (h);
	\draw [very thick, <-] (f).. controls (0.4, -0.1) and (1.6, -0.1) .. (h);
	\draw [very thick, <-] (e) -- (d);
	\draw [very thick, <-] (f).. controls (0, 0.6) and (0.6, 1) .. (e);
\end{tikzpicture}
\caption{A  bifiltered chain complex $C$ without a simultaneously vertically and horizontally simplified basis. Note that $\varepsilon(C)=0$.}
\label{fig:nobasis}
\end{figure}
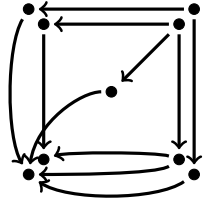

We now describe certain properties of the invariant $\varepsilon$ which will be of use in understanding knot concordance.
From \cite[Proposition 3.6]{Homcables}, we have the following:
\begin{enumerate}
	\item If $K$ is slice, then $\varepsilon(K)=0$.
	\item If $\varepsilon(K)=0$, then $\tau(K)=0$.
	\item $\varepsilon(-K) = - \varepsilon(K).$
	\item \begin{enumerate}
			\item If $\varepsilon(K_1)=\varepsilon(K_2)$, then $\varepsilon(K_1 \# K_2)=\varepsilon(K_1)$.
			\item If $\varepsilon(K_1)=0$, then $\varepsilon(K_1 \# K_2) = \varepsilon(K_2)$.
		\end{enumerate}
\end{enumerate}
We will use these properties to define a certain totally ordered group, onto which the knot concordance group surjects.

\begin{definition}
\label{defn:mathfrakC}
Let $\mathfrak{C}$ be the set of $\Z \oplus \Z$-filtered, free, finitely generated chain complexes over $\F[U, U^{-1}]$ such that
\begin{enumerate}
	\item The complex obtained by reversing the roles of the two filtrations is filtered homotopic to the original.
	\item The homology of any column has rank one.
	\item The total homology of the complex is isomorphic to $\F[U, U^{-1}]$.
\end{enumerate}
\end{definition}

\noindent Note that given a complex $C \in \mathfrak{C}$, we may define $\tau(C)$ and $\varepsilon(C)$ as above. We define the group $\cCFKa$ to be
\[ \cCFKa = \mathfrak{C} / \sim \]
where the operation, which we denote $+$, is induced by tensor product and 
\[ C_1 \sim C_2 \quad \textup{if and only if} \quad \varepsilon(C_1 \otimes C_2^*)=0 .\] 
The inverse of $[C]$ is given by $[C^*]$. We define the group $\cCFK$ to be the subgroup of $\cCFKa$ generated by the knot Floer complexes of knots in $S^3$, that is,
\[ \cCFK = \{ \CFKi(K) \mid K \subset S^3 \} / \sim. \]
There is a group homomorphism
\[ \cC \rightarrow \cCFK \]
defined to be
\[ [K] \mapsto [\CFKi(K)].\]
The fact that this is a group homomorphism follows from the behavior of $\CFKi$ under connected sum and mirroring, together with the fact that $\varepsilon(K)=0$ if $K$ is slice.

As described in the introduction, the group $\cCFKa$ (and hence $\cCFK \subset \cCFKa$ as well) is totally ordered, with the ordering given by
\[ [C_1] > [C_2] \quad \textup{if and only if} \quad \varepsilon(C_1 \otimes C_2^*)=1.\]
In the subsequent sections, we develop more refined invariants for understanding the order structure of $\cCFK$.

\section{Numerical invariants associated to the knot Floer complex}
\label{sec:numinvts}
In this section, we will recall and extend certain numerical invariants associated to $[ \CFKi(K) ]$ from \cite{Homsmooth}. The invariant $\varep$ is defined by examining whether or not certain natural maps on subquotient complexes of $\CFKi$, corresponding to certain cobordisms, vanish on homology. We extend this construction to a broader family of maps. The (non-)vanishing of these maps may be interpreted as the existence of certain arrows of prescribed lengths in the graphical depiction of $\CFKi$. Indeed, we will use these invariants to find certain bases for $\CFKi$ which will be useful for computations.

If $\varep(K)=1$, then the map on homology induced by
\[ C\{i=0\} \rightarrow C\{ \min (i, j-\tau)=0 \} \]
is trivial, and one may consider the map $H_s$ on homology induced by
\[ C\{i=0\} \rightarrow C\{ \min (i, j-\tau) = 0, i \leq s \}, \]
where $s$ is a non-negative integer. When $\varep(K) = 1$, we define $a_1(K)$ to be
\[ a_1(K) = \min \{ s \ | \ H_s \textup{ trivial} \}.\]
Note that when $s=0$, the map $H_s$ is non-trivial by the definition of $\tau$, and that when $s$ is sufficiently large and $\varep(K)=1$, the map $H_s$ is trivial. Hence $a_1(K)$ is well-defined.

One may now consider the map $H_{a_1, s}$ induced on homology by
\[ C\{i=0\} \rightarrow C\big\{ \{ \min(i, j-\tau)=0, i \leq a_1 \} \cup \{ i=a_1, \tau-s \leq j < \tau \} \big\}, \]
where again $s$ is a non-negative integer. The map $H_{a_1, 0}$ is trivial since it agrees with the map $H_{a_1}$. Define
\[ a_2(K) = \min \{ s \ | \ H_{a_1, s} \textup{ non-trivial} \}. \]
Note that $a_2$ may be undefined, as there is no reason why $H_{a_1, s}$ must become non-trivial at some point.

We may define higher $a_n$'s inductively. Assume that $a_k(K)$ exists for $1 \leq k \leq n$. Let $s$ be a non-negative integer. Define $S(a_1, \ldots, a_n, s)$ to be
\[ S(a_1, \ldots, a_n, s)=\left\{
	\begin{array}{ll}
		S(a_1, \ldots, a_n) \cup \{ j=\tau-a_\even^n, \ a_\odd^{n} < i \leq a_\odd^{n} +s \} & \textup{ if } n \textup{ is even}\\
		S(a_1, \ldots, a_n) \cup \{ i=a_\odd^n, \ \tau-a_\even^{n}-s \leq j < \tau - a_\even^{n} \} & \textup{ if }  n \textup{ is odd},
	\end{array} \right.\]
where $S(s)=\{ \min (i, j-\tau) = 0, i \leq s \}$, and 
\[ a_\even^n = \sum_{\substack{2 \leq k \leq n \\ k \ \even}} a_k \qquad  \qquad  a_\odd^n = \sum_{\substack{1 \leq k \leq n \\ k \ \odd}} a_k. \]
See Figure \ref{fig:S}.
Then consider the map $H_{a_1, \ldots, a_n, s}$ induced on homology by
\[ C\{i=0\} \rightarrow C\{ S(a_1, \ldots, a_n, s)\}, \]
where we first quotient $C \{ i=0 \}$ by $C\{ i=0, j < \tau\}$, and then include into $C\{ S(a_1, \ldots, a_n, s) \}$. Define
\[ a_{n+1}(K)=\left\{
	\begin{array}{ll}
	\min \{ s \ | \ H_{a_1, \ldots, a_{n}, s} \textup{ trivial} \} &\textup{ if } n \textup{ is even} \\
	\min \{ s \ | \ H_{a_1, \ldots, a_{n}, s} \textup{ non-trivial} \} &\textup{ if } n \textup{ is odd.}
\end{array} \right.\]
Suppose $a_n(K)$ is well-defined but $a_{n+1}(K)$ is not. Then let
\[ \ba^+(K) = \big(a_1(K), a_2(K), \ldots, a_n(K)\big). \]

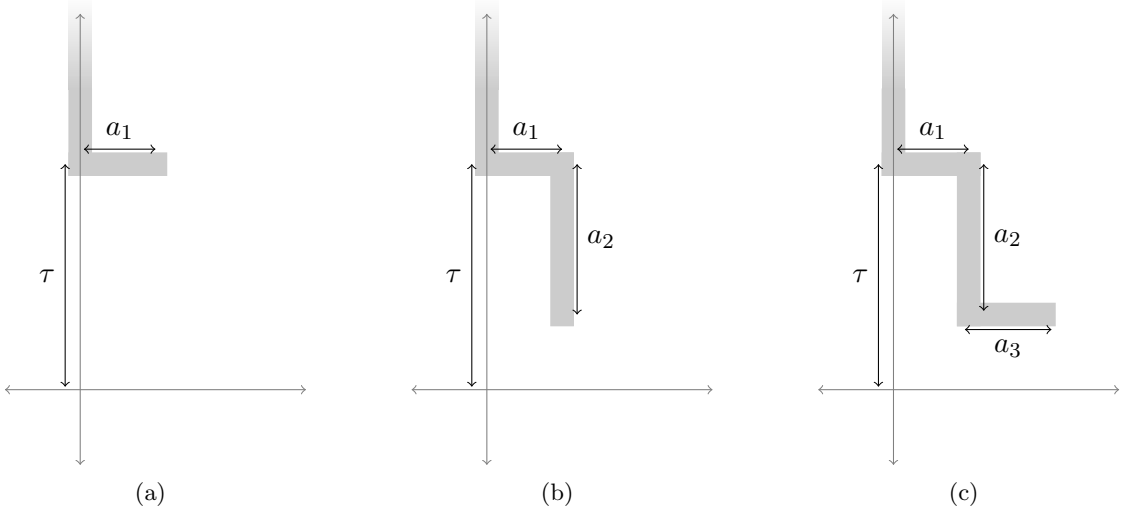
\begin{figure}[htb!]
\vspace{5pt}
\labellist
\small \hair 2pt
\endlabellist
\centering
\subfigure[]{
\begin{tikzpicture}
	\filldraw[black!20!white] (-0.15, 2.85) rectangle (0.15, 4);
	\filldraw[black!20!white] (-0.15, 2.85) rectangle (1.15, 3.15);
	\shade[top color=white, bottom color=black!20!white] (-0.15, 4) rectangle (0.15, 5.25);
	\begin{scope}[thin, gray]
		\draw [<->] (-1, 0) -- (3, 0);
		\draw [<->] (0, -1) -- (0, 5);
	\end{scope}
	\draw [<->] (-0.2, 0.04) -- (-0.2, 3) node [midway, left] {$\tau$};
	\draw [<->] (0.05, 3.2) -- (1, 3.2) node [midway, above] {$a_1$};
\end{tikzpicture}
}
\hspace{25pt}
\subfigure[]{
\begin{tikzpicture}
	\filldraw[black!20!white] (-0.15, 2.85) rectangle (0.15, 4);
	\filldraw[black!20!white] (-0.15, 2.85) rectangle (1.15, 3.15);
	\shade[top color=white, bottom color=black!20!white] (-0.15, 4) rectangle (0.15, 5.25);
	\filldraw[black!20!white] (0.85, 0.85) rectangle (1.15, 3.15);
	\begin{scope}[thin, gray]
		\draw [<->] (-1, 0) -- (3, 0);
		\draw [<->] (0, -1) -- (0, 5);
	\end{scope}
	\draw [<->] (-0.2, 0.04) -- (-0.2, 3) node [midway, left] {$\tau$};
	\draw [<->] (0.05, 3.2) -- (1, 3.2) node [midway, above] {$a_1$};
	\draw [<->] (1.2, 1) -- (1.2, 3) node [midway, right] {$a_2$};	
\end{tikzpicture}
}
\hspace{25pt}
\subfigure[]{
\begin{tikzpicture}
	\filldraw[black!20!white] (-0.15, 2.85) rectangle (0.15, 4);
	\filldraw[black!20!white] (-0.15, 2.85) rectangle (1.15, 3.15);
	\shade[top color=white, bottom color=black!20!white] (-0.15, 4) rectangle (0.15, 5.25);
	\filldraw[black!20!white] (0.85, 0.85) rectangle (1.15, 3.15);
	\filldraw[black!20!white] (0.85, 0.85) rectangle (2.15, 1.15);
	\begin{scope}[thin, gray]
		\draw [<->] (-1, 0) -- (3, 0);
		\draw [<->] (0, -1) -- (0, 5);
	\end{scope}
	\draw [<->] (-0.2, 0.04) -- (-0.2, 3) node [midway, left] {$\tau$};
	\draw [<->] (0.05, 3.2) -- (1, 3.2) node [midway, above] {$a_1$};
	\draw [<->] (1.2, 1.05) -- (1.2, 3) node [midway, right] {$a_2$};	
	\draw [<->] (.95, 0.8) -- (2.1, 0.8) node [midway, below] {$a_3$};	
\end{tikzpicture}
}

\caption{Left, the complex $C\{ S(a_1)\}$. Center, the complex $C\{ S(a_1, a_2) \}$. Right, the complex $C \{ S(a_1, a_2, a_3)\}$.}
\label{fig:S}
\end{figure}

\begin{proposition}
The tuple $\ba^+(K)$ is an invariant of the class $[\CFKi(K)]$.
\end{proposition}

\begin{proof}
The result follows from the proof of Lemma 6.1 in \cite{Homsmooth}. Specifically, if $\CFKi(K_1)$ and $\CFKi(K_2)$ are $\varep$-equivalent, then 
\[ [\CFKi(K_1)]=[\CFKi(K_2)]=[\CFKi(K_1) \otimes \CFKi(K_2)^* \otimes \CFKi(K_2)] . \]
Since $\varep(K_1 \# -K_2)=0$, by \cite[Lemma 3.3]{Homcables} there exists a basis for $\CFKi(K_1) \otimes \CFKi(K_2)^*$ with a distinguished element, say $x_0$, with no incoming or outgoing horizontal or vertical arrows. Similarly, since $\varep(K_2 \# -K_2)=0$, there is a distinguished element, say $y_0$, with no incoming or outgoing horizontal or vertical arrows. Then we may compute $\ba^+(K_1 \# -K_2 \# K_2)$ by considering either
\[ \{x_0\} \otimes \CFKi(K_2) \quad \textup{or} \quad \CFKi(K_1) \otimes \{y_0\}, \]
the former giving $\ba^+(K_1)$ and the latter giving $\ba^+(K_2)$.
\end{proof}

The tuple $\ba^+(K)$ implies the existence of a particular basis for $\CFKi(K)$ with certain arrows of certain lengths, as described in the following proposition.

\begin{proposition}
\label{prop:basis}
Suppose $\boldsymbol{a}^+(K)=(a_1, a_2, \ldots, a_n)$. Then there exists a basis for $\{ x_i \}$ over $\F[U, U^{-1}]$ for $\CFKi(K)$ with a subcollection of basis elements $x_0, x_1, \ldots, x_n$ such that
\begin{enumerate}
	\item There is a horizontal arrow of length $a_i$ from $x_i$ to $x_{i-1}$ from $i$ odd, $1\leq i \leq n$.
	\item There is a vertical arrow of length $a_i$ from $x_{i-1}$ to $x_{i}$ for $i$ even, $2 \leq i \leq n$.
	\item There are no other horizontal or vertical arrows to or from $x_0, x_1, \ldots, x_{n-1}$.
	\item If $n$ is odd, then there are no other horizontal arrows to or from $x_n$. If $n$ is even, there are no other vertical arrows to or from $x_n$.
\end{enumerate}
\end{proposition}

\noindent A basis satisfying the conclusion of Proposition \ref{prop:basis} is called \emph{locally simplified}. See Figure \ref{fig:basis}. In order to prove this proposition, we first prove a sequence of lemmas.

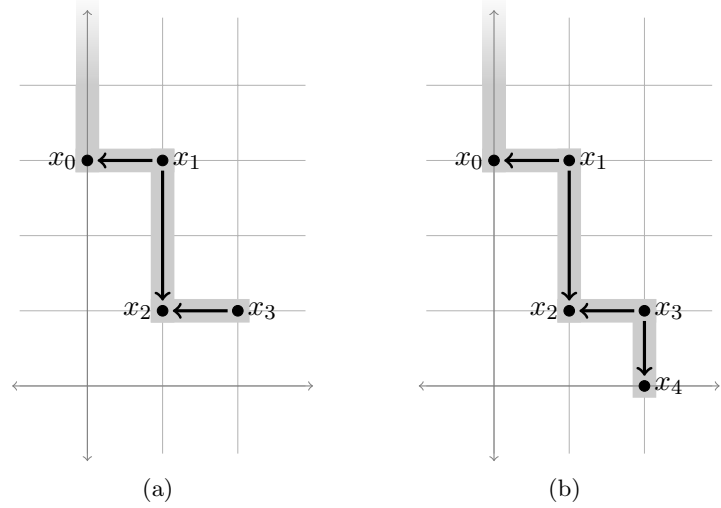
\begin{figure}[htb!]
\vspace{5pt}
\labellist
\small \hair 2pt
\endlabellist
\centering
\subfigure[]{
\begin{tikzpicture}
	\draw[step=1, black!30!white, very thin] (-0.9, -0.9) grid (2.9, 4.9);
	\filldraw[black!20!white] (-0.15, 2.85) rectangle (0.15, 4);
	\filldraw[black!20!white] (-0.15, 2.85) rectangle (1.15, 3.15);
	\shade[top color=white, bottom color=black!20!white] (-0.15, 4) rectangle (0.15, 5.25);
	\filldraw[black!20!white] (0.85, 0.85) rectangle (1.15, 3.15);
	\filldraw[black!20!white] (0.85, 0.85) rectangle (2.15, 1.15);
	\begin{scope}[thin, gray]
		\draw [<->] (-1, 0) -- (3, 0);
		\draw [<->] (0, -1) -- (0, 5);
	\end{scope}
	\filldraw (0, 3) circle (2pt) node[] (x_0){};
	\filldraw (1, 3) circle (2pt) node[] (x_1){};
	\filldraw (1, 1) circle (2pt) node[] (x_2){};
	\filldraw (2, 1) circle (2pt) node[] (x_3){};
	\draw [very thick, <-] (x_0) -- (x_1);
	\draw [very thick, <-] (x_2) -- (x_1);
	\draw [very thick, <-] (x_2) -- (x_3);
	\node [left] at (x_0) {$x_0$};
	\node [right] at (x_1) {$x_1$};
	\node [left] at (x_2) {$x_2$};
	\node [right] at (x_3) {$x_3$};
\end{tikzpicture}
}
\hspace{25pt}
\subfigure[]{
\begin{tikzpicture}
	\draw[step=1, black!30!white, very thin] (-0.9, -0.9) grid (2.9, 4.9);
	\filldraw[black!20!white] (-0.15, 2.85) rectangle (0.15, 4);
	\filldraw[black!20!white] (-0.15, 2.85) rectangle (1.15, 3.15);
	\shade[top color=white, bottom color=black!20!white] (-0.15, 4) rectangle (0.15, 5.25);
	\filldraw[black!20!white] (0.85, 0.85) rectangle (1.15, 3.15);
	\filldraw[black!20!white] (0.85, 0.85) rectangle (2.15, 1.15);
	\filldraw[black!20!white] (1.85, -0.15) rectangle (2.15, 1.15);
	\begin{scope}[thin, gray]
		\draw [<->] (-1, 0) -- (3, 0);
		\draw [<->] (0, -1) -- (0, 5);
	\end{scope}
	\filldraw (0, 3) circle (2pt) node[] (x_0){};
	\filldraw (1, 3) circle (2pt) node[] (x_1){};
	\filldraw (1, 1) circle (2pt) node[] (x_2){};
	\filldraw (2, 1) circle (2pt) node[] (x_3){};
	\filldraw (2, 0) circle (2pt) node[] (x_4){};
	\draw [very thick, <-] (x_0) -- (x_1);
	\draw [very thick, <-] (x_2) -- (x_1);
	\draw [very thick, <-] (x_2) -- (x_3);
	\draw [very thick, <-] (x_4) -- (x_3);
	\node [left] at (x_0) {$x_0$};
	\node [right] at (x_1) {$x_1$};
	\node [left] at (x_2) {$x_2$};
	\node [right] at (x_3) {$x_3$};
	\node [right] at (x_4) {$x_4$};
\end{tikzpicture}
\label{subfig:neven}
}
\caption{Left, part of the basis in Proposition \ref{prop:basis} for $n$ odd and $\ba(K)=(1, 2, 1)$. The shaded region depicts $C\{S(1, 2, 1)\}$. Right, for $n$ even and $\ba(K)=(1, 2, 1, 1)$; the shaded region depicts $C\{S(1, 2, 1, 1)\}$.}
\label{fig:basis}
\end{figure}

\begin{lemma}
\label{lem:basissub}
Suppose  $\boldsymbol{a}^+(K)=(a_1, a_2, \ldots, a_n)$ where $n$ is even. There exists a basis $\{x_j\}$ for $\CFKi(K)$ over $\F[U, U^{-1}]$ with a subset of basis elements
\[ x_0, x_1, \ldots, x_n \]
with the following non-zero differentials:
\begin{align*}
	&\partial^\horz x_i = x_{i-1}  && 1\leq i \leq n-1,\ i \textup{ odd} \\
	&\partial^\vert x_i = x_{i+1} && 1\leq i \leq n-1,\ i \textup{ odd}.
\end{align*}
\end{lemma}

\begin{proof}
We may assume that our complex is reduced; that is, that the differential structs lowers filtration. Since $\boldsymbol{a}^+(K)=(a_1, a_2, \ldots, a_n)$, there exists a chain $\bfx$ in $C\{S(a_1, \ldots, a_n)\}$ whose boundary in $C\{S(a_1, \ldots, a_n)\}$ is $x_0+x_n$, where $x_0$ is some element supported in filtration level $(0, \tau)$ and $x_n$ is in filtration level $(a_1+a_3+ \dots + a_{n-1}, \tau-a_2-a_4-\dots -a_n)$. The idea behind this proof is to decompose $\bfx$ into a sum of elements with filtration level corresponding to the local maxima of $C\{S(a_1, \ldots, a_n)\}$, and then to use a change of basis to obtain the desired differentials.

We decompose $\bfx$ as $x_{-1}+x_1+x_3+ \dots x_{n-1}$ where $x_{-1}$ has filtration level $(0, m)$ for some $m>\tau$ and each other $x_i$ (where $i$ is odd) has filtration level $(a_1+a_3+\dots +a_i, \tau-a_2-a_4-\dots -a_{i-1})$; see Figure \ref{subfig:basissubfirst}. Note that a priori $\bfx$ may include terms in other, lower filtration levels, but such terms may be absorbed into the $x_i$. 

We may assume that $\partial^\vert x_{n-1}=x_n$, i.e., that there is a vertical arrow from $x_{n-1}$ to $x_n$. We have that $\partial^\horz x_{n-1} = \partial^\vert x_{n-3}$ in $C\{S(a_1, \ldots, a_n)\}$. Let $x_{n-2}$ denote $\partial^\horz x_{n-1}$ plus the part of $\partial^\vert x_{n-3}$ that has vertical filtration level less than $\partial^\horz x_{n-1}$; see Figures \ref{subfig:basissubmid} and \ref{subfig:basissublast}. Then there is a horizontal arrow from $x_{n-1}$ to $x_{n-2}$ and a vertical arrow from $x_{n-3}$ to $x_{n-2}$. (Note that there may also be diagonal arrows from $x_{n-1}$ or $x_{n-3}$.) Continue in this manner, concluding with replacing $x_0$ with $\partial^\horz x_1$ plus the part of $\partial^\vert x_{-1}$ that has filtration level less than $\partial^\horz x_1$. Note that the new $x_0$ still generates the vertical homology, and has no incoming vertical arrows. See Figure \ref{subfig:basissublast}. 

We have formed a  subcomplex of $C\{S(a_1, \ldots, a_n)\}$ consisting of 
\[ x_0, x_1, \ldots, x_n. \]
The non-zero horizontal and vertical differentials are 
\begin{align*}
	&\partial^\horz x_i = x_{i-1}  && 1\leq i \leq n-1,\ i \textup{ odd} \\
	&\partial^\vert x_i = x_{i+1} && 1\leq i \leq n-1,\ i \textup{ odd},
\end{align*}
and we may extend $x_0, \ldots x_n$ to a basis $\{x_j\}$ over $\F[U, U^{-1}]$ for $\CFKi(K)$.
\end{proof}

\begin{figure}[htb!]
\vspace{5pt}
\labellist
\small \hair 2pt
\endlabellist
\centering
\subfigure[]{
\begin{tikzpicture}
	\filldraw[black!20!white] (-0.15, 2.85) rectangle (0.15, 4);
	\filldraw[black!20!white] (-0.15, 2.85) rectangle (1.15, 3.15);
	\shade[top color=white, bottom color=black!20!white] (-0.15, 4) rectangle (0.15, 5.25);
	\filldraw[black!20!white] (0.85, 0.85) rectangle (1.15, 3.15);
	\filldraw[black!20!white] (0.85, 0.85) rectangle (2.15, 1.15);
	\filldraw[black!20!white] (1.85, -0.15) rectangle (2.15, 1.15);
	\begin{scope}[thin, gray]
		\draw [<->] (-1, 0) -- (3, 0);
		\draw [<->] (0, -1) -- (0, 5);
	\end{scope}
	\filldraw (0, 4) circle (2pt) node[] (x_-1){};
	\filldraw (0, 3) circle (2pt) node[] (x_0){};
	\filldraw (1, 3) circle (2pt) node[] (x_1){};
	\filldraw (2, 1) circle (2pt) node[] (x_3){};
	\filldraw (2, 0) circle (2pt) node[] (x_4){};
	\filldraw (2, 0) circle (2pt) node[] (x_4){};
	\node [left] at (x_-1) {$x_{-1}$};
	\node [left] at (x_0) {$x_0$};
	\node [right] at (x_1) {$x_1$};
	\node [right] at (x_3) {$x_3$};
	\node [right] at (x_4) {$x_4$};
\end{tikzpicture}
\label{subfig:basissubfirst}
}
\hspace{25pt}
\subfigure[]{
\begin{tikzpicture}
	\filldraw[black!20!white] (-0.15, 2.85) rectangle (0.15, 4);
	\filldraw[black!20!white] (-0.15, 2.85) rectangle (1.15, 3.15);
	\shade[top color=white, bottom color=black!20!white] (-0.15, 4) rectangle (0.15, 5.25);
	\filldraw[black!20!white] (0.85, 0.85) rectangle (1.15, 3.15);
	\filldraw[black!20!white] (0.85, 0.85) rectangle (2.15, 1.15);
	\filldraw[black!20!white] (1.85, -0.15) rectangle (2.15, 1.15);
	\begin{scope}[thin, gray]
		\draw [<->] (-1, 0) -- (3, 0);
		\draw [<->] (0, -1) -- (0, 5);
	\end{scope}
	\filldraw (0, 4) circle (2pt) node[] (x_-1){};
	\filldraw (0.1, 3.1) circle (2pt) node[] (w){};	
	\filldraw (-0.1, 2.9) circle (2pt) node[] (x_0){};
	\filldraw (1, 3) circle (2pt) node[] (x_1){};
	\filldraw (1, 1) circle (2pt) node[] (x_2){};
	\filldraw (2, 1) circle (2pt) node[] (x_3){};
	\filldraw (2, 0) circle (2pt) node[] (x_4){};
	\filldraw (1, 0) circle (2pt) node[] (z){};
	\filldraw (0, 1) circle (2pt) node[] (v){};
	\filldraw (0, 2) circle (2pt) node[] (u){};
	\draw [very thick, <-] (w) -- (x_-1);
	\draw [very thick, <-] (w) .. controls (0.4, 3.1) and (0.6, 3.1) ..  (x_1);
	\draw [very thick, <-] (x_0) .. controls (0.4, 2.8) and (0.6, 2.8) ..  (x_1);
	\draw [very thick, <-] (x_2) -- (x_1);
	\draw [very thick, <-] (x_2) -- (x_3);
	\draw [very thick, <-] (x_4) -- (x_3);
	\draw [very thick, <-] (z) .. controls (0.6, 1.2) and (0.6, 1.8) ..  (x_1);
	\draw [very thick, <-] (v) .. controls (0.7, 0.8) and (1.3, 0.8) ..  (x_3);
	\draw [very thick, <-] (u) .. controls (-0.25, 2.8) and (-0.25, 3.2) ..  (x_-1);
	\node [above right] at (w) {$w$};
	\node [left] at (x_-1) {$x_{-1}$};
	\node [below right] at (x_0) {$x_0$};
	\node [right] at (x_1) {$x_1$};
	\node [above right] at (x_2) {$y$};
	\node [right] at (x_3) {$x_3$};
	\node [right] at (x_4) {$x_4$};
	\node [right] at (z) {$z$};
	\node [left] at (v) {$v$};
	\node [left] at (u) {$u$};
\end{tikzpicture}
\label{subfig:basissubmid}
}
\hspace{25pt}
\subfigure[]{
\begin{tikzpicture}
	\filldraw[black!20!white] (-0.15, 2.85) rectangle (0.15, 4);
	\filldraw[black!20!white] (-0.15, 2.85) rectangle (1.15, 3.15);
	\shade[top color=white, bottom color=black!20!white] (-0.15, 4) rectangle (0.15, 5.25);
	\filldraw[black!20!white] (0.85, 0.85) rectangle (1.15, 3.15);
	\filldraw[black!20!white] (0.85, 0.85) rectangle (2.15, 1.15);
	\filldraw[black!20!white] (1.85, -0.15) rectangle (2.15, 1.15);
	\begin{scope}[thin, gray]
		\draw [<->] (-1, 0) -- (3, 0);
		\draw [<->] (0, -1) -- (0, 5);
	\end{scope}
	\filldraw (0, 4) circle (2pt) node[] (x_-1){};
	\filldraw (0.1, 3.1) circle (2pt) node[] (w){};	
	\filldraw (-0.1, 2.9) circle (2pt) node[] (x_0){};
	\filldraw (1, 3) circle (2pt) node[] (x_1){};
	\filldraw (1, 1) circle (2pt) node[] (x_2){};
	\filldraw (2, 1) circle (2pt) node[] (x_3){};
	\filldraw (2, 0) circle (2pt) node[] (x_4){};
	\filldraw (1, 0) circle (2pt) node[] (x){};
	\filldraw (0, 1) circle (2pt) node[] (v){};
	\filldraw (0, 2) circle (2pt) node[] (u){};
	\draw [very thick, <-] (w) -- (x_-1);
	\draw [very thick, <-] (x_0) .. controls (0.4, 2.9) and (0.6, 2.9) ..  (x_1);
	\draw [very thick, <-] (x_2) -- (x_1);
	\draw [very thick, <-] (x_2) -- (x_3);
	\draw [very thick, <-] (x_4) -- (x_3);
	\draw [very thick, <-] (x) -- (x_3);
	\draw [very thick, <-] (u) -- (x_1);
	\draw [very thick, <-] (v) -- (x_1);
	\draw [very thick, <-] (u) .. controls (-0.3, 2.8) and (-0.3, 3.2) ..  (x_-1);
	\node [above right] at (w) {$w$};
	\node [left] at (x_-1) {$x_{-1}$};
	\node [left] at (x_0) {$x'_0$};
	\node [right] at (x_1) {$x_1$};
	\node [left] at (x_2) {$x_2$};
	\node [right] at (x_3) {$x_3$};
	\node [right] at (x_4) {$x_4$};
	\node [right] at (x) {$z$};
	\node [left] at (v) {$v$};
	\node [left] at (u) {$u$};
\end{tikzpicture}
\label{subfig:basissublast}
}
\caption{An example of the change of basis from Lemma \ref{lem:basissub} for $n=4$. At right, $x_2=v+y+z$ and $x'_0=x_0+u+w$. Note that this figure does not include all of the generators of $\CFKi(K)$.}
\label{fig:basissub}
\end{figure}
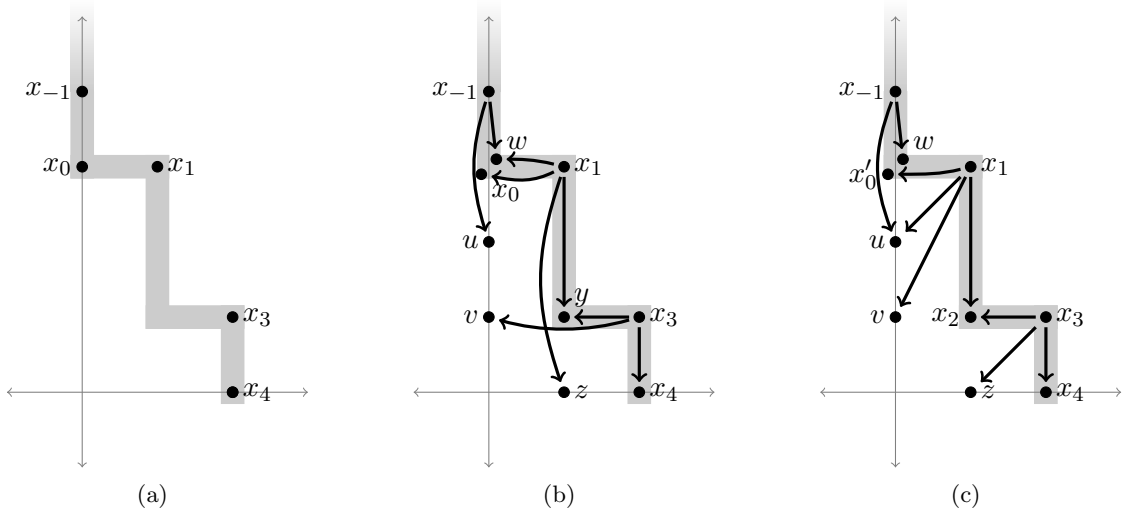

\begin{lemma}
\label{lem:arrowremoval}
Suppose  $\boldsymbol{a}^+(K)=(a_1, a_2, \ldots, a_n)$ where $n$ is even. Let $1\leq i \leq n-1$, $i$ odd. Given the basis from Lemma \ref{lem:basissub} with vertical arrows from $x_i$ to $x_{i+1}$ and horizontal arrows from $x_i$ to $x_{i-1}$, we may perform a change of basis to remove all other incoming vertical arrows to $x_{i+1}$ and all other incoming horizontal arrows to $x_{i-1}$.
\end{lemma}

\begin{proof}
We will perform a change of basis to remove any unwanted vertical or horizontal arrows. There is a vertical arrow from $x_{n-1}$ to $x_n$.
We will remove all other vertical arrows to $x_n$ using the following changes of basis:

\begin{enumerate}
	\item If there is a vertical arrow to $x_n$ from an element $y$ with vertical filtration level strictly greater than that of $x_{n-1}$, then we replace $y$ with $y+x_{n-1}$. This is a filtered change of basis since the filtration level of $y$ is greater than that of $x_{n-1}$. See Figures \ref{subfig:removeverta} and \ref{subfig:removevertb}.
	\item Now suppose there is a vertical arrow to $x_n$ from an element $y$ of vertical filtration less than or equal to that of $x_{n-1}$. 
	\begin{enumerate}
		\item If there is another vertical arrow from $y$ to some element $z$ contained in $C\{S(a_1, \ldots, a_n)\}$, then we replace $z$ with $z+x_n$. This is a filtered change of basis since $z$ is contained in $C\{S(a_1, \ldots, a_n)\}$ and thus must have filtration level greater than or equal to that of $x_n$. See Figures \ref{subfig:removevertc} and \ref{subfig:removevertd}.
		\item If such a vertical arrow does not exist, then there must be a horizontal arrow from $y$ to some element $z$. This is because the element $x_n$ is not in the image of of the boundary on $C\{S(a_1, \ldots, a_n)\}$ as the map $H_{a_1, \ldots a_n}$ is non-trivial. We then replace $y$ with $y+x_{n-1}$. Note that $y$ must have the same filtration level as $x_{n-1}$, since there is a horizontal arrow from $y$ to $z$ in $C\{S(a_1, \ldots, a_n)\}$. See Figures \ref{subfig:removeverte} and \ref{subfig:removevertf}.
	\end{enumerate}
\end{enumerate}
In this manner, we may remove all other incoming vertical arrows to $x_n$. 

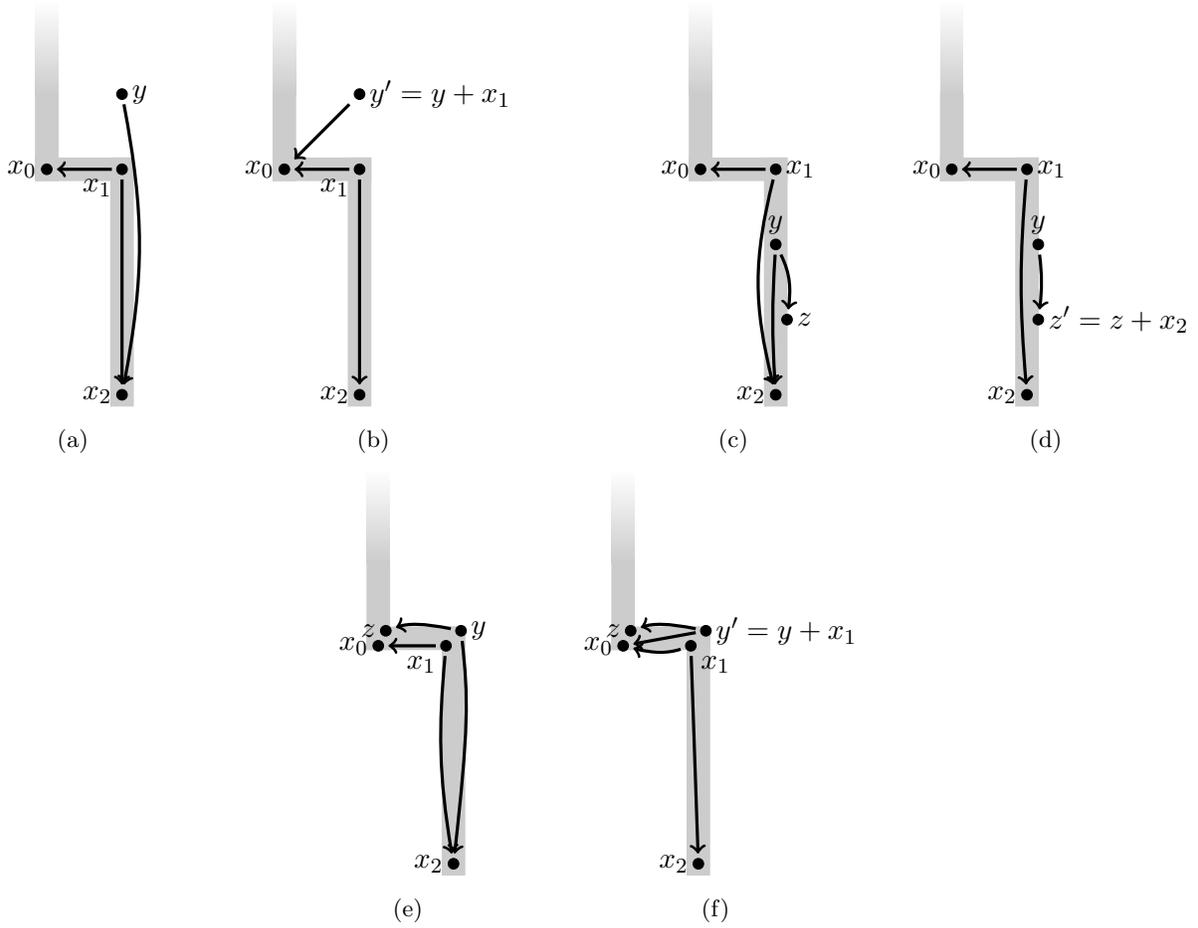
\begin{figure}[htb!]
\vspace{5pt}
\labellist
\small \hair 2pt
\endlabellist
\centering
\subfigure[]{
\begin{tikzpicture}
	\filldraw[black!20!white] (-0.15, 2.85) rectangle (0.15, 4);
	\filldraw[black!20!white] (-0.15, 2.85) rectangle (1.15, 3.15);
	\shade[top color=white, bottom color=black!20!white] (-0.15, 4) rectangle (0.15, 5.25);
	\filldraw[black!20!white] (0.85, -0.15) rectangle (1.15, 3.15);
	\filldraw (0, 3) circle (2pt) node[] (x_0){};
	\filldraw (1, 3) circle (2pt) node[] (x_1){};
	\filldraw (1, 0) circle (2pt) node[] (x_2){};
	\filldraw (1, 4) circle (2pt) node[] (y){};
	\draw [very thick, <-] (x_0) -- (x_1);
	\draw [very thick, <-] (x_2) -- (x_1);
	\draw [very thick, <-] (x_2)  .. controls (1.3, 1.6) and (1.3, 2.4) ..   (y);
	\node [left] at (x_0) {$x_0$};
	\node [below left] at (x_1) {$x_1$};
	\node [left] at (x_2) {$x_2$};
	\node [right] at (y) {$y$};
\end{tikzpicture}
\label{subfig:removeverta}
}
\hspace{15pt}
\subfigure[]{
\begin{tikzpicture}
	\filldraw[black!20!white] (-0.15, 2.85) rectangle (0.15, 4);
	\filldraw[black!20!white] (-0.15, 2.85) rectangle (1.15, 3.15);
	\shade[top color=white, bottom color=black!20!white] (-0.15, 4) rectangle (0.15, 5.25);
	\filldraw[black!20!white] (0.85, -0.15) rectangle (1.15, 3.15);
	\filldraw (0, 3) circle (2pt) node[] (x_0){};
	\filldraw (1, 3) circle (2pt) node[] (x_1){};
	\filldraw (1, 0) circle (2pt) node[] (x_2){};
	\filldraw (1, 4) circle (2pt) node[] (y){};
	\draw [very thick, <-] (x_0) -- (x_1);
	\draw [very thick, <-] (x_2) -- (x_1);
	\draw [very thick, <-] (x_0) -- (y);
	\node [left] at (x_0) {$x_0$};
	\node [below left] at (x_1) {$x_1$};
	\node [left] at (x_2) {$x_2$};
	\node [right] at (y) {$y'=y+x_1$};
\end{tikzpicture}
\label{subfig:removevertb}
}
\hspace{35pt}
\subfigure[]{
\begin{tikzpicture}
	\filldraw[black!20!white] (-0.15, 2.85) rectangle (0.15, 4);
	\filldraw[black!20!white] (-0.15, 2.85) rectangle (1.15, 3.15);
	\shade[top color=white, bottom color=black!20!white] (-0.15, 4) rectangle (0.15, 5.25);
	\filldraw[black!20!white] (0.85, -0.15) rectangle (1.15, 3.15);
	\filldraw (0, 3) circle (2pt) node[] (x_0){};
	\filldraw (1, 3) circle (2pt) node[] (x_1){};
	\filldraw (1, 0) circle (2pt) node[] (x_2){};
	\filldraw (1, 2) circle (2pt) node[] (y){};
	\filldraw (1.15, 1) circle (2pt) node[] (z){};
	\draw [very thick, <-] (x_0) -- (x_1);
	\draw [very thick, <-] (x_2)   .. controls (0.7, 1.2) and (0.7, 1.8) ..  (x_1);
	\draw [very thick, <-] (x_2) .. controls (0.95, 0.8) and (0.95, 1.2) ..  (y);
	\draw [very thick, <-] (z) .. controls (1.2, 1.4) and (1.2, 1.6) ..  (y);
	\node [left] at (x_0) {$x_0$};
	\node [right] at (x_1) {$x_1$};
	\node [left] at (x_2) {$x_2$};
	\node [above] at (y) {$y$};
	\node [right] at (z) {$z$};
\end{tikzpicture}
\label{subfig:removevertc}
}
\hspace{15pt}
\subfigure[]{
\begin{tikzpicture}
	\filldraw[black!20!white] (-0.15, 2.85) rectangle (0.15, 4);
	\filldraw[black!20!white] (-0.15, 2.85) rectangle (1.15, 3.15);
	\shade[top color=white, bottom color=black!20!white] (-0.15, 4) rectangle (0.15, 5.25);
	\filldraw[black!20!white] (0.85, -0.15) rectangle (1.15, 3.15);
	\filldraw (0, 3) circle (2pt) node[] (x_0){};
	\filldraw (1, 3) circle (2pt) node[] (x_1){};
	\filldraw (1, 0) circle (2pt) node[] (x_2){};
	\filldraw (1.15, 2) circle (2pt) node[] (y){};
	\filldraw (1.15, 1) circle (2pt) node[] (z){};
	\draw [very thick, <-] (x_0) -- (x_1);
	\draw [very thick, <-] (x_2)   .. controls (0.9, 1.2) and (0.9, 1.8) ..  (x_1);
	\draw [very thick, <-] (z) .. controls (1.2, 1.4) and (1.2, 1.6) ..  (y);
	\node [left] at (x_0) {$x_0$};
	\node [right] at (x_1) {$x_1$};
	\node [left] at (x_2) {$x_2$};
	\node [above] at (y) {$y$};
	\node [right] at (z) {$z'=z+x_2$};
\end{tikzpicture}
\label{subfig:removevertd}
}
\hspace{15pt}
\subfigure[]{
\begin{tikzpicture}
	\filldraw[black!20!white] (-0.15, 2.85) rectangle (0.15, 4);
	\filldraw[black!20!white] (-0.15, 2.85) rectangle (1.15, 3.15);
	\shade[top color=white, bottom color=black!20!white] (-0.15, 4) rectangle (0.15, 5.25);
	\filldraw[black!20!white] (0.85, -0.15) rectangle (1.15, 3.15);
	\filldraw (0, 2.9) circle (2pt) node[] (x_0){};
	\filldraw (0.9, 2.9) circle (2pt) node[] (x_1){};
	\filldraw (1, 0) circle (2pt) node[] (x_2){};
	\filldraw (1.1, 3.1) circle (2pt) node[] (y){};
	\filldraw (0.1, 3.1) circle (2pt) node[] (z){};
	\draw [very thick, <-] (x_0) -- (x_1);
	\draw [very thick, <-] (x_2)   .. controls (0.8, 1.2) and (0.8, 1.8) ..  (x_1);
	\draw [very thick, <-] (x_2) .. controls (1.2, 1.8) and (1.2, 2.2) ..  (y);
	\draw [very thick, <-] (z) .. controls (0.4, 3.2) and (0.6, 3.2) ..  (y);
	\node [left] at (x_0) {$x_0$};
	\node [below left] at (x_1) {$x_1$};
	\node [left] at (x_2) {$x_2$};
	\node [right] at (y) {$y$};
	\node [left] at (z) {$z$};
\end{tikzpicture}
\label{subfig:removeverte}
}
\hspace{15pt}
\subfigure[]{
\begin{tikzpicture}
	\filldraw[black!20!white] (-0.15, 2.85) rectangle (0.15, 4);
	\filldraw[black!20!white] (-0.15, 2.85) rectangle (1.15, 3.15);
	\shade[top color=white, bottom color=black!20!white] (-0.15, 4) rectangle (0.15, 5.25);
	\filldraw[black!20!white] (0.85, -0.15) rectangle (1.15, 3.15);
	\filldraw (0, 2.9) circle (2pt) node[] (x_0){};
	\filldraw (0.9, 2.9) circle (2pt) node[] (x_1){};
	\filldraw (1, 0) circle (2pt) node[] (x_2){};
	\filldraw (1.1, 3.1) circle (2pt) node[] (y){};
	\filldraw (0.1, 3.1) circle (2pt) node[] (z){};
	\draw [very thick, <-] (x_0) .. controls (0.4, 2.8) and (0.6, 2.8) ..  (x_1);
	\draw [very thick, <-] (x_2) -- (x_1);
	\draw [very thick, <-] (x_0) -- (y);
	\draw [very thick, <-] (z)   .. controls (0.4, 3.2) and (0.6, 3.2) ..  (y);
	\node [left] at (x_0) {$x_0$};
	\node [below right] at (x_1) {$x_1$};
	\node [left] at (x_2) {$x_2$};
	\node [right] at (y) {$y'=y+x_1$};
	\node [left] at (z) {$z$};
\end{tikzpicture}
\label{subfig:removevertf}
}
\caption{Examples of removing unwanted vertical arrows. Here, $n=2$.}
\label{fig:removingverticalarrows}
\end{figure}

We now proceed to remove all horizontal arrows to $x_{n-2}$ other than the horizontal arrow from $x_{n-1}$ with similar changes of bases.
\begin{enumerate}
	\item If there is a horizontal arrow from an element $y$ to $x_{n-2}$ where $y$ has horizontal filtration greater than $x_{n-1}$, then we replace $y$ with $y+x_{n-1}$. This is a filtered change of basis since $x_{n-1}$ has filtration level less than that of $y$.
	\item
	If there is a horizontal arrow from an element $y$ of horizontal filtration less than or equal to that of $x_{n-1}$, then we again consider the other outgoing arrows from $y$. 
	\begin{enumerate}
		\item If there is another horizontal arrow from $y$ to some element $z$ contained in $C\{S(a_1, \ldots, a_n)\}$, then we replace $z$ with $z+x_{n-2}$. This is a filtered change of basis since $z$ is contained in $C\{S(a_1, \ldots, a_n)\}$ and thus has filtration level greater than or equal to that of $x_{n-2}$.
		\item If there is no other horizontal arrow from $y$ to an element contained in $C\{S(a_1, \ldots, a_n)\}$, then there must be a vertical arrow from $y$ to some element $z$, because the map  $H_{a_1, \ldots a_n}$ is non-trivial. In this case, replace $y$ with $y+x_{n-1}$ and $z$ with $z+x_n$. This is a filtered change of basis because $y$ must have the same filtration level as $x_{n-1}$ and $z$ must have filtration level greater than or equal to that of $x_n$ because $z$ is contained in $C\{S(a_1, \ldots, a_n)\}$.
	\end{enumerate}
\end{enumerate}	
All remaining unwanted horizontal arrows to $x_{n-2}$ may be removed in this manner. 

Continuing in this manner, we may remove all of the unwanted vertical and horizontal arrows, completing the proof of the lemma.
\end{proof}

\begin{proof}[Proof of Proposition \ref{prop:basis}]
Suppose  $\boldsymbol{a}^+(K)=(a_1, a_2, \ldots, a_n)$ where $n$ is even. We apply Lemmas \ref{lem:basissub} and \ref{lem:arrowremoval} to obtain a basis $\{x_i\}$ for $\CFKi(K)$ 
with a subset of basis elements
\[ x_0, x_1, \ldots, x_n \]
with the following non-zero differentials:
\begin{align*}
	&\partial^\horz x_i = x_{i-1}  && 1\leq i \leq n-1,\ i \textup{ odd} \\
	&\partial^\vert x_i = x_{i+1} && 1\leq i \leq n-1,\ i \textup{ odd},
\end{align*}
such that there are no other incoming vertical arrows to $x_{i+1}$ or incoming horizontal arrows to $x_{i-1}$, for $1\leq i \leq n-1$, $i$ odd.

For $i$ odd, $1 \leq i \leq n-1$, there is a unique outgoing vertical arrow from $x_{i}$, which is the unique incoming vertical arrow to $x_{i+1}$, and since $\partial^\vert$ is a differential, i.e., squares to zero, we may conclude that there are no outgoing vertical arrows from $x_{i+1}$, nor any incoming vertical arrows to $x_i$. Similarly, for $i$ odd, $1 \leq i \leq n-1$, we may conclude that there are no outgoing horizontal arrows from $x_{i-1}$, nor any incoming horizontal arrows to $x_i$.

The proof in the case that $n$ is  odd follows similarly and is left to the reader.
\end{proof}

In the case where $n$ is odd and $a_n$ is well-defined but $a_{n+1}$ is not, we can further modify the basis, as described in the following lemma.

\begin{lemma}
Suppose $\boldsymbol{a}^+(K)=(a_1, a_2, \ldots, a_n)$, where $n$ is odd, and $a_n$ is well-defined but $a_{n+1}$ is not. Then we may assume that for the basis $\{ x_i\}$ from the conclusion of  Proposition \ref{prop:basis}, we have that $\partial^\vert x_n = 0$.
\end{lemma}

\begin{proof}
We begin with the basis $\{x_i\}$ from the conclusion of Proposition \ref{prop:basis} and modify it so that $\partial^\vert x_n = 0$. Since $a_{n+1}$ is undefined, the map on homology induced by 
\[ C\{i=0\} \rightarrow C\{ S(a_1, \ldots, a_n, s) \} \]
is trivial for all $s$, and so we have that $x_0$ is nullhomologous in $C\{ S(a_1, \ldots, a_n, s) \}$ for all $s$. 

Based on the differentials described in Proposition \ref{prop:basis}, we have that $x_0$ is homologous to $\partial^\vert x_n$ in $C\{S(a_1, \dots, a_n, s)\}$ for all $s$. Thus $\partial^\vert x_n$ is nullhomologous in $C\{ S(a_1, \ldots, a_n, s) \}$, i.e., there exists some $\bfy$ such that the boundary of $\bfy$ in $C\{ S(a_1, \ldots, a_n, s) \}$ equals $\partial^\vert x_n$. We decompose this chain $\bfy$ as the sum of $y_1+y_3+ \dots + y_n$ where $A(y_i) \leq A(x_i)$ and $w(y_i) \leq w(x_i)$, where $A$ and $w$ denote the two filtrations on $\CFKi(K)$. Note that $\partial^\vert y_n=\partial^\vert x_n$.

By replacing $x_i$ with $x_i+y_i$ for $i$ odd, and following the procedure in the proof of Lemmas \ref{lem:basissub} and \ref{lem:arrowremoval}, we obtain the desired basis.
\end{proof}

When $n$ is odd and $a_n$ is well-defined but $a_{n+1}$ is not, we note the following.
Since the class $[x_0]$ generates the vertical homology (which has rank one), the class $[x_n]$ must be zero in vertical homology, and since $x_n\in \ker \partial^\vert$, it must also be in the image of $\partial^\vert$. We would like to measure the length of the incoming arrow to $x_n$, as follows. Consider the short exact sequence 
\[ 0 \rightarrow C\{S(a_1, \ldots, a_n-1) \} \rightarrow C\{S(a_1, \ldots, a_n) \} \rightarrow C\{ i=a^n_\textup{odd}, j=\tau - a^{n-1}_\textup{even} \} \rightarrow 0, \]
where $a^{n}_\textup{odd}=a_1 + a_3 + \dots + a_{n}$ and $a^{n-1}_\textup{even}=a_2 + a_4 + \dots a_{n-1}$. 
Let $\delta$ denote the connecting homomorphism of the associated long exact sequence. Define
\[ a'_{n+1}(K) = - \min \left\{ A(y)-A(\partial^\vert y) \mid \delta([ \partial^\vert y]) = [x_0] \in H_*(C\{S(a_1, \ldots, a_n-1) \}) \right\}. \]
See Figure \ref{subfig:nodd}. 

Now suppose that $n$ is even and again let $\{ x_i\}_{i=0}^n$ is the basis from Proposition \ref{prop:basis}, i.e., the invariant $a_n$ is well-defined but $a_{n+1}$ is not. Since $a_{n+1}$ is not defined, it follows that $x_n$ is not in the image of $\partial^\horz$. Consider the inclusion
\[ \iota: C\{ i=a^{n-1}_\odd, j=\tau-a^n_\even \} \rightarrow C\{S(a_1, \ldots, a_n) \}. \]
Define
\[ a'_{n+1}(K) = -\min \big\{ w(y)-w(\partial^\horz y) \mid \iota_*([y]) = [x_0] \in H_*(C\{S(a_1, \ldots, a_n) \}) \big\}. \]
See Figure \ref{subfig:neven}. In effect, $a'_{n+1}$ measures the length of the outgoing horizontal arrow from $x_n$. Note that it is possible for $\partial^\horz x_n=0$ (in which case $x_n$ is a generator of the horizontal homology), in which case $a'_{n+1}$ is not defined.

\begin{figure}[htb!]
\vspace{5pt}
\labellist
\small \hair 2pt
\endlabellist
\centering
\subfigure[]{
\begin{tikzpicture}
	\filldraw[black!20!white] (-0.15, 2.85) rectangle (0.15, 4);
	\filldraw[black!20!white] (-0.15, 2.85) rectangle (1.15, 3.15);
	\shade[top color=white, bottom color=black!20!white] (-0.15, 4) rectangle (0.15, 5.25);
	\filldraw[black!20!white] (0.85, 0.85) rectangle (1.15, 3.15);
	\filldraw[black!20!white] (0.85, 0.85) rectangle (3.15, 1.15);
	\begin{scope}[thin, gray]
		\draw [<->] (-1, 0) -- (4, 0);
		\draw [<->] (0, -1) -- (0, 5);
	\end{scope}
	\filldraw (0, 3) circle (2pt) node[] (x_0){};
	\filldraw (1, 3) circle (2pt) node[] (x_1){};
	\filldraw (1, 1) circle (2pt) node[] (x_2){};
	\filldraw (3, 1) circle (2pt) node[] (x_3){};
	\filldraw (3, 4) circle (2pt) node[] (y){};
	\draw [very thick, <-] (x_0) -- (x_1);
	\draw [very thick, <-] (x_2) -- (x_1);
	\draw [very thick, <-] (x_2) -- (x_3);
	\draw [very thick, <-] (x_3) -- (y);
	\node [left] at (x_0) {$x_0$};
	\node [right] at (x_1) {$x_1$};
	\node [left] at (x_2) {$x_2$};
	\node [right] at (x_3) {$x_3$};
	\node [right] at (y) {$y$};
\end{tikzpicture}
\label{subfig:nodd}
}
\hspace{25pt}
\subfigure[]{
\begin{tikzpicture}
	\filldraw[black!20!white] (-0.15, 2.85) rectangle (0.15, 4);
	\filldraw[black!20!white] (-0.15, 2.85) rectangle (1.15, 3.15);
	\shade[top color=white, bottom color=black!20!white] (-0.15, 4) rectangle (0.15, 5.25);
	\filldraw[black!20!white] (0.85, 0.85) rectangle (1.15, 3.15);

	\begin{scope}[thin, gray]
		\draw [<->] (-1, 0) -- (2, 0);
		\draw [<->] (0, -1) -- (0, 5);
	\end{scope}
	\filldraw (0, 3) circle (2pt) node[] (x_0){};
	\filldraw (1, 3) circle (2pt) node[] (x_1){};
	\filldraw (1, 1) circle (2pt) node[] (x_2){};
	\filldraw (0, 1) circle (2pt) node[] (y){};
	\draw [very thick, <-] (x_0) -- (x_1);
	\draw [very thick, <-] (x_2) -- (x_1);
	\draw [very thick, <-] (y) -- (x_2);
	\node [left] at (x_0) {$x_0$};
	\node [right] at (x_1) {$x_1$};
	\node [right] at (x_2) {$x_2$};
\end{tikzpicture}
\label{subfig:neven}
}
%
\caption{Example of the bases from Lemma \ref{lem:basis}. Left, $n$ odd, and right, $n$ even.} 
\label{}
\end{figure}
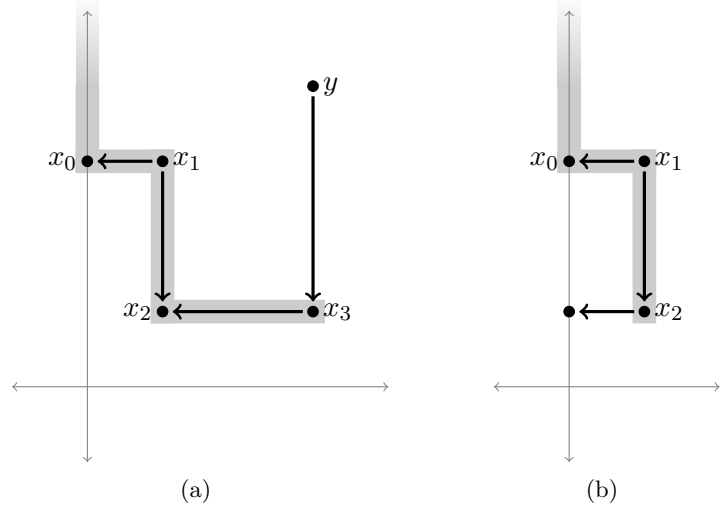

We have the following strengthening of Proposition \ref{prop:basis} when $a'_{n+1}(K)$ is well-defined:
\begin{lemma}
\label{lem:basis}
Suppose $\ba^+(K) = (a_1, a_2, \ldots, a_n)$ and $a'_{n+1}(K) = a'_{n+1}$ is well-defined. Then there exists a basis for $\{ x_i \}$ over $\F[U, U^{-1}]$ for $\CFKi(K)$ with a subcollection of basis elements $x_0, x_1, \ldots, x_{n+1}$ such that
\begin{enumerate}
	\item There is a horizontal arrow of length $a_i$ from $x_i$ to $x_{i-1}$ from $i$ odd, $1\leq i \leq n$.
	\item There is a vertical arrow of length $a_i$ from $x_{i-1}$ to $x_{i}$ for $i$ even, $2 \leq i \leq n$.
	\item If $n$ is odd, there is a vertical arrow of length $|a'_{n+1}|$ from $x_{n+1}$ to $x_n$. If $n$ is even, there is a horizontal arrow of length $|a'_{n+1}|$ from $x_n$ to $x_{n+1}$.
	\item There are no other horizontal or vertical arrows to or from $x_0, x_1, \ldots, x_n$.
	\item If $n$ is odd, then there are no other vertical arrows to or from $x_{n+1}$. If $n$ is even, there are no other horizontal arrows to or from $x_{n+1}$.
\end{enumerate}
\end{lemma}

\begin{proof}
We first consider the case where $n$ is odd. Let $a'_{n+1}=a'_{n+1}(K)$. Let $\bfy$ be a chain such that $A(\bfy) - A(\partial^\vert \bfy) = |a'_{n+1}|$ and $\delta ([\partial^\vert \bfy] ) = [x_0] \in H_*(C\{S(a_1, \ldots, a_n-1) \})$. Apply Proposition \ref{prop:basis}, letting $x_n=\partial^\vert \bfy$, to obtain a locally simplified basis. 

Let $y$ be a basis element appearing in $\bfy$ with maximal Alexander grading. (Note that $\bfy$ is in general a linear combination of basis elements.) We replace $y$ with $\bfy$, and denote this new basis element by $x_{n+1}$. There is now a vertical arrow from $x_{n+1}$ to $x_n$. 



We now proceed to remove all other vertical arrows into $x_n$. Suppose there is an incoming arrow from a basis element $z$ to $x_n$. There are two possible cases:
\begin{enumerate}
	\item If $A(z) \geq A(x_{n+1})$, then we replace $z$ with $z+x_{n+1}$. This is a filtered change of basis, and removes the arrow from $z$ to $x_n$.
	\item If $A(z) < A(x_{n+1})$, then there must be another vertical arrow from $z$ to some element $w$. This is because $x_n$ is not in the image of any elements with Alexander grading less than that of $y$. Moreover, $A(w) \geq A(x_{n})$, since otherwise $[\partial^\vert z] = [x_n] \in C\{ i=a^n_\textup{odd}, j=\tau - a^{n-1}_\textup{even} \}$. Thus, we may replace $w$ with $w+x_n$, which removes the vertical arrow from $z$ to $x_n$.
\end{enumerate}
In this manner, we may remove all other incoming vertical arrows to $x_n$.

Since $\partial^\vert x_{n+1}=x_n$, it follows from $\partial^\vert \circ \partial^\vert=0$ that there are no outgoing vertical arrows from $x_n$. Furthermore, since there are no other vertical arrows to $x_n$, it again follows from $\partial^\vert \circ \partial^\vert=0$ that there are no incoming arrows to $x_{n+1}$. This completes the proof of the lemma for $n$ odd.

The case for $n$ even follows similarly and is left to the reader.
\end{proof}

We now show that in certain special cases, the values of $a'_{n+1}$ are limited by the terms in the sequence $\ba^+(K)$. 

\begin{lemma}
\label{lem:a1=1}
If $a_1=1$, then $a_2$ is defined.
\end{lemma}

\begin{proof}
Suppose, for contradiction, that $a_2$ is not defined. Then $a'_2$ must be defined. Consider the basis as in the conclusion of Lemma \ref{lem:basis}, i.e., there is a horizontal arrow of length one from $x_1$ to $x_0$, a vertical arrow of length $|a'_2|$ from $x_2$ to $x_1$, and no other vertical or horizontal arrows to $x_0$ or $x_1$. Furthermore, there are no other vertical arrows from $x_2$. Then $x_0$ appears exactly in $\partial^2 x_2$, since the only other way for $x_0$ to appear in $\partial^2 x_2$ would be for there to be a sequence of two diagonal arrows, first from $x_2$ to some element $y$ and then from $y$ to $x_0$. But this is impossible, since $a_1=1$. See Figure \ref{fig:a1=1}.
\end{proof}

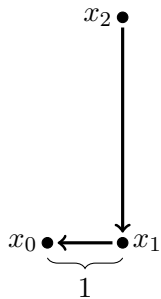
\begin{figure}[htb!]
\vspace{5pt}
\labellist
\small \hair 2pt
\endlabellist
\centering
\begin{tikzpicture}
	\filldraw (0, 0) circle (2pt) node[] (x_0){};
	\filldraw (1, 0) circle (2pt) node[] (x_1){};
	\filldraw (1, 3) circle (2pt) node[] (x_2){};
	\draw [very thick, <-] (x_0) -- (x_1);
	\draw [very thick, <-] (x_1) -- (x_2);
	\node [left] at (x_0) {$x_0$};
	\node [right] at (x_1) {$x_1$};
	\node [left] at (x_2) {$x_2$};
	\draw [decorate,decoration={brace,amplitude=4pt,mirror}] (0, -0.2) -- (1, -0.2) node [midway, below, yshift=-4pt] {$1$};
\end{tikzpicture}
\caption{The basis used in the proof of Lemma \ref{lem:a1=1}.}
\label{fig:a1=1}
\end{figure}

We write $(1, n)^k$ to denote
\[ (\overbrace{1, n, \ldots, 1, n}^{k}). \]

\begin{lemma}
\label{lem:sequences}
Let $k$, $\ell$, and $n$ be positive integers. If $\ba^+(K)=((1, n)^k, 1)$ and $a'_{2k+2}(K)=a'_{2k+2}$, then
\[ a'_{2k+2}< -n. \]
Similarly, if $\ba^+(K)=((1, n)^k, (n, 1)^\ell)$and $a'_{2k+2\ell+1}(K)=a'_{2k+2\ell+1}$, then
\[ a'_{2k+2\ell+1} < -n. \]
\end{lemma}

\begin{proof}
We first consider the case $\ba^+(K)=((1, n)^k, 1)$ and $a'_{2k+2}(K)=a'_{2k+2}$. We may assume that we have a basis that satisfies the conclusion of Lemma \ref{lem:basis}.

We proceed by contradiction. Assume that $-n \leq a'_{2k+2} < 0$. See Figure \ref{subfig:contradictiona}. Then 
\[ \partial^\vert x_{2k+2} = x_{2k+1} \quad \textup{and} \quad \partial^\horz x_{2k+1}= x_{2k}. \]
Since $\partial^2=0$, we need to find another appearance of $x_{2k}$ in $\partial^2 x_{2k+2}$ to cancel with the term above. The only horizontal arrow to $x_{2k}$ is from $x_{2k+1}$, and the only vertical arrow to $x_{2k}$ is from $x_{2k-1}$. However, by assumption $A(x_{2k-1})> A(x_{2k+2})$, so there cannot be an arrow from $x_{2k+2}$ to $x_{2k-1}$.

The only remaining possibility is for there to be a diagonal arrow to $x_{2k}$ from some element $y$. But in order for this to contribute to $\partial^2 x_{2k+2}$, the element $y$ must be in the same column as $x_{2k+2}$, since $a_{2k+1}=1$. But the only outgoing vertical arrow from $x_{2k+2}$ is to $x_{2k+1}$.

Thus, if $-n \leq a'_{2k+2} < 0$, it is impossible for $\partial^2$ to equal zero, completing the proof of the first statement of the lemma.

The second statement follows similarly and is left to the reader. See Figure \ref{subfig:contradictionb}.
\end{proof}

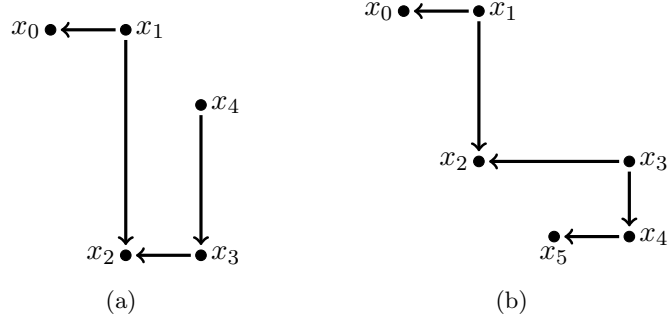
\begin{figure}[htb!]
\vspace{5pt}
\labellist
\small \hair 2pt
\endlabellist
\centering
\subfigure[]{
\begin{tikzpicture}
	\filldraw (0, 4) circle (2pt) node[] (x_0){};
	\filldraw (1, 4) circle (2pt) node[] (x_1){};
	\filldraw (1, 1) circle (2pt) node[] (x_2){};
	\filldraw (2, 1) circle (2pt) node[] (x_3){};
	\filldraw (2, 3) circle (2pt) node[] (x_4){};
	\draw [very thick, <-] (x_0) -- (x_1);
	\draw [very thick, <-] (x_2) -- (x_1);
	\draw [very thick, <-] (x_2) -- (x_3);
	\draw [very thick, <-] (x_3) -- (x_4);
	\node [left] at (x_0) {$x_0$};
	\node [right] at (x_1) {$x_1$};
	\node [left] at (x_2) {$x_2$};
	\node [right] at (x_3) {$x_3$};
	\node [right] at (x_4) {$x_4$};
\end{tikzpicture}
\label{subfig:contradictiona}
}
\hspace{25pt}
\subfigure[]{
\begin{tikzpicture}
	\filldraw (0, 3) circle (2pt) node[] (x_0){};
	\filldraw (1, 3) circle (2pt) node[] (x_1){};
	\filldraw (1, 1) circle (2pt) node[] (x_2){};
	\filldraw (3, 1) circle (2pt) node[] (x_3){};
	\filldraw (3, 0) circle (2pt) node[] (x_4){};
	\filldraw (2, 0) circle (2pt) node[] (x_5){};
	\draw [very thick, <-] (x_0) -- (x_1);
	\draw [very thick, <-] (x_2) -- (x_1);
	\draw [very thick, <-] (x_2) -- (x_3);
	\draw [very thick, <-] (x_4) -- (x_3);
	\draw [very thick, <-] (x_5) -- (x_4);
	\node [left] at (x_0) {$x_0$};
	\node [right] at (x_1) {$x_1$};
	\node [left] at (x_2) {$x_2$};
	\node [right] at (x_3) {$x_3$};
	\node [right] at (x_4) {$x_4$};
	\node [below] at (x_5) {$x_5$};
\end{tikzpicture}
\label{subfig:contradictionb}
}
\caption{Inadmissible complexes from the proof of Lemma \ref{lem:sequences}.}
\label{}
\end{figure}

These invariants are almost sufficient for the purpose of this paper. In the case where $a'_{n+1} = -1$, we must continue one step further. If $\ba^+(K)=(a_1, \ldots, a_n)$, $a'_{n+1}(K)=-1$, and $n$ is odd, define
\[ a'_{n+2} = -\min \big\{ w(y) -w(\partial^\horz y) \mid y \in Y \big\}, \]
where
\[ Y = \big\{ y \mid A(y)-A(\partial^\vert y)=|a'_{n+1}|,\ \delta ( [\partial^\vert y] ) = [x_0] \in H_*(C(S\{a_1 , \ldots, a_n-1)\}) \big\}. \]
In effect, this measures the length of the outgoing horizontal arrow from $x_{n+1}$. If $\ba^+(K)=(a_1, \ldots, a_n)$, $a'_{n+1}(K)=-1$, and $n$ is even, define
\[ a'_{n+2}(K) = -\min \big\{ A(z) - A(\partial^\vert z) \mid \partial^\vert z = \partial^\horz y, y \in Y \big\}, \]
where
\[ Y= \big\{ y \mid w(y)-w(\partial^\horz y) = |a'_{n+1}|, \ \iota_*([y]) = [x_0] \in H_*(C\{S(a_1, \ldots, a_n) \} )  \big\}. \]
In this case, $a'_{n+2}$ measures the length of the incoming vertical arrow to $x_{n+1}$.

In this situation, we again have a lemma about a related basis:

\begin{lemma}
\label{lem:basis2}
If $\ba^+(K)=(a_1, \ldots, a_n)$, $a'_{n+1}(K)=-1$, and $a'_{n+2}(K)=a'_{n+2}$, then there exists a basis $\{x_i\}$ over $\F[U, U^{-1}]$ for $\CFKi(K)$ with basis elements $x_0, \ldots, x_{n+2}$ such that
\begin{enumerate}
	\item There is a horizontal arrow of length $a_i$ from $x_i$ to $x_{i-1}$ from $i$ odd, $1\leq i \leq n$.
	\item There is a vertical arrow of length $a_i$ from $x_{i-1}$ to $x_{i}$ for $i$ even, $2 \leq i \leq n$.
	\item If $n$ is odd, there is a vertical arrow of length $|a'_{n+1}|$ from $x_{n+1}$ to $x_n$. If $n$ is even, there is a horizontal arrow of length $|a'_{n+1}|$ from $x_n$ to $x_{n+1}$.
	\item If $n$ is odd, there is a horizontal arrow of length $|a'_{n+2}|$ from $x_{n+1}$ to $x_{n+2}$. If $n$ is even, there is a vertical arrow of length $|a'_{n+2}|$ from $x_{n+2}$ to $x_{n+1}$.
	\item There are no other horizontal or vertical arrows to or from $x_0, x_1, \ldots, x_{n+1}$.
	\item If $n$ is odd, then there are no other horizontal arrows to or from $x_{n+2}$. If $n$ is even, there are no other vertical arrows to or from $x_{n+2}$.
\end{enumerate}
\end{lemma}

\begin{proof}
The result follows as in the proof of Lemma \ref{lem:basis}.
\end{proof}

\noindent We again have certain restrictions on the possible values for $a'_{n+1}$ and $a'_{n+2}$.

\begin{lemma}
\label{lem:sequences2}
Suppose $\ba^+(\CFKi(K))=((1, n)^k)$ and $a'_{2k+1}=-1$. Then $a'_{2k+2}(K)=a'_{2k+2}$ is well-defined, and satisfies
\[ 0< |a'_{2k+2}| < n. \]
\end{lemma}

\begin{proof}
The result follows as in the proof of Lemma \ref{lem:sequences}.
\end{proof}

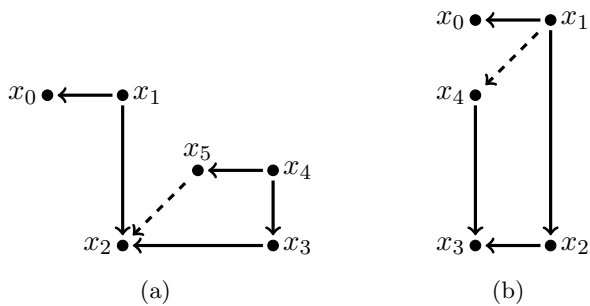
\begin{figure}[htb!]
\vspace{5pt}
\labellist
\small \hair 2pt
\endlabellist
\centering
\subfigure[]{
\begin{tikzpicture}
	\filldraw (0, 3) circle (2pt) node[] (x_0){};
	\filldraw (1, 3) circle (2pt) node[] (x_1){};
	\filldraw (1, 1) circle (2pt) node[] (x_2){};
	\filldraw (3, 1) circle (2pt) node[] (x_3){};
	\filldraw (3, 2) circle (2pt) node[] (x_4){};
	\filldraw (2, 2) circle (2pt) node[] (x_5){};
	\draw [very thick, <-] (x_0) -- (x_1);
	\draw [very thick, <-] (x_2) -- (x_1);
	\draw [very thick, <-] (x_2) -- (x_3);
	\draw [very thick, <-] (x_3) -- (x_4);
	\draw [very thick, <-] (x_5) -- (x_4);
	\draw[dashed, very thick, <-] (x_2) -- (x_5);
	\node [left] at (x_0) {$x_0$};
	\node [right] at (x_1) {$x_1$};
	\node [left] at (x_2) {$x_2$};
	\node [right] at (x_3) {$x_3$};
	\node [right] at (x_4) {$x_4$};
	\node [above] at (x_5) {$x_5$};
\end{tikzpicture}
\label{}
}
\hspace{25pt}
\subfigure[]{
\begin{tikzpicture}
	\filldraw (0, 4) circle (2pt) node[] (x_0){};
	\filldraw (1, 4) circle (2pt) node[] (x_1){};
	\filldraw (1, 1) circle (2pt) node[] (x_2){};
	\filldraw (0, 1) circle (2pt) node[] (x_3){};
	\filldraw (0, 3) circle (2pt) node[] (x_4){};
	\draw [very thick, <-] (x_0) -- (x_1);
	\draw [very thick, <-] (x_2) -- (x_1);
	\draw [very thick, <-] (x_3) -- (x_2);
	\draw [very thick, <-] (x_3) -- (x_4);
	\draw[dashed, very thick, <-] (x_4) -- (x_1);
	\node [left] at (x_0) {$x_0$};
	\node [right] at (x_1) {$x_1$};
	\node [right] at (x_2) {$x_2$};
	\node [left] at (x_3) {$x_3$};
	\node [left] at (x_4) {$x_4$};
\end{tikzpicture}
\label{}
}
\caption{Examples of complexes that satisfy the conclusion of Lemma \ref{lem:sequences2}. We include diagonal arrows that will ensure that $\partial^2=0$, although the lemma does not contain an explicit statement about diagonal arrows.}
\label{}
\end{figure}

We write $\ba(K)$ to denote the sequence of $a_i$'s, followed by $a'_i$'s, if defined. More precisely, let 
\[ \ba(K)=\left\{
	\begin{array}{ll}
	(a_1(K), \ldots, a_n(K))  &\textup{ if neither } a_{n+1}(K) \textup{ nor } a'_{n+1}(K) \textup{ is defined} \\
	(a_1(K), \ldots, a_n(K), a'_{n+1}(K)) &\textup{ if } a'_{n+1}(K) \textup{ is defined and not equal to } -1 \\
	(a_1(K), \ldots, a_n(K), a'_{n+1}(K), a'_{n+2}(K)) &\textup{ if } a'_{n+1}(K)=-1.
\end{array} \right.\]
Often, the first $k$ elements of $\ba(K)$ are sufficient to determine information about the Archimedean equivalence class of $[\CFKi(K)] \in \cCFK$.

We would like to distinguish when the tuple $\ba(K)$ completely determines the $\varepsilon$-equivalence class of $K$. For example, if $K$ is an $L$-space knot (i.e., a knot admitting a positive $L$-space surgery), then the tuple $\ba(K)$ completely describes $\CFKi(K)$, as in Figure \ref{subfig:T_3,4}. More generally, if $\ba(K)=(a_1(K), \ldots, a_n(K))$ and in the locally simplified basis associated to $\ba(K)$, the element $x_n$ is a generator of the horizontal homology, then $\ba(K)$ completely determines the $\varepsilon$-equivalence class of $K$. In this case, we use square brackets to denote $\ba(K)$, i.e., 
\[ \ba(K) = [a_1, \ldots, a_n]. \]
The invariant $\ba(K)$ and the associated locally simplified basis will be of use in better understanding the order type of $\cCFK$.

Note that since the tuple $\ba(K)$ is an invariant of the complex $\CFKi(K)$, we may similarly define $\ba(C)$, where $C \in \mathfrak{C}$, i.e., $C$ satisfies the same symmetry and rank properties that the knot Floer complex does. 

\section{Computations}
\label{sec:comp}
In this section, we study the order type of $\cCFKa$ and find an infinite family  of Archimedean equivalence classes in $\cCFKa$ satisfying Property A.

\begin{proposition}
\label{prop:ArchA}
The Archimedean equivalence class of $C_n=[1, n, n, 1]$, $n\geq 2$, satisfies Property A. Namely, any element in $\cCFKa$ that is Archimedean equivalent to $C_n$ is of the form $k \cdot C_n + C'$ where $|C'| \ll C_n$ and $k \in \Z$.
\end{proposition}

We begin by recalling some relationships between the tuple $\ba(C)$ and the ordering on $\cCFKa$. 
\begin{lemma}[{\cite[Lemmas 6.3 and 6.4]{Homsmooth}}]
\label{lem:agg}
If $\ba(J)=(a_1, \ldots)$ and $\ba(K)=(b_1, \ldots)$ with $a_1 > b_1 >0$, then
\[[\CFKi(J)] \ll [\CFKi(K)] . \]
If $\ba(J)=(a_1, a_2, \ldots)$ and $\ba(K)=(b_1, b_2, \ldots)$ with $a_1=b_1>0$ and $a_2 > b_2>0$, then
\[[\CFKi(J)] \gg [\CFKi(K)] . \]
\end{lemma}

\noindent In other words, we have that
\begin{align*}
\label{eqn:lemma6.3}	(1, \ldots) &\gg (n, \ldots) \textup{ for } n \geq 2\\
	(1, m, \ldots) &\gg (1, n, \ldots) \textup{ for } m>n. \nonumber
\end{align*}

We will use the notation $k[1, n, n, 1]$ to denote
\[ \overbrace{[1, n, n, 1] + \dots + [1, n, n, 1]}^{k \textup{ times}}.\]
Recall the notation
\[ (1, n)^k = (\overbrace{1, n, \ldots, 1, n}^{k}). \]
By \cite[Lemma 3.1]{HancockHomNewman} we have that
\[ k[1, n, n, 1] = [(1, n)^k, (n, 1)^k]. \]

Fix an integer $n\geq2$, and let $C_n=[1, n, n, 1]$. We will show that for a positive element $C$ in $\cCFKa$, exactly one of the following holds:
\begin{itemize}
	\item $C$ is equal to a multiple of $C_n$, i.e., $C=k C_n$ for some $k \in \N$
	\item $C \gg C_n$
	\item $C = k C_n + C'$, for some $k \in \N \cup \{0\}$ and $C' \in \cCFKa$ with $0<|C'| \ll C_n$.
\end{itemize}
Our first step is to determine the possible forms of $C$ when $C$ is not a multiple of $C_n$.

\begin{lemma}
\label{lem:forms}
Let $C$ be a positive element of $\cCFKa$. If $C$ is not a multiple of $[1, n, n, 1]$ for some integer $n \geq 2$, then $\ba(C)$ is of one of the following forms:
\begin{enumerate}[label=(\alph*)]
	\item \label{item:mgeq2} $(m, \ldots)$ for some $m \geq 2$
	\item \label{item:11} $(1, 1, \ldots)$
	\item \label{item:1n1m} $((1, n)^k, 1, m, \ldots)$ for some $m$ such that $m< -n$, $0<m<n$, or $n < m$  
	\item \label{item:1nm} $((1, n)^k, m, \ldots)$ for some $m$ such that $m < -1$, $1< m <n$, or $n < m$
	\item \label{item:1n-1m} $((1, n)^k, -1, m, \ldots)$ for some $m$ such that $-n<m<0$ 
	\item \label{item:1nn1m} $((1, n)^k, (n, 1)^\ell, m, \ldots)$ for some $1\leq \ell<k$ and $m$ such that $m<-n$, $0<m<n$, or $n < m$
	\item \label{item:1nn1nm} $((1, n)^k, (n,1)^\ell, n, m, \ldots)$ for some $\ell<k$ and $m$ such that $m<0$ or $1 < m$
	\item \label{item:1nn1c} $((1, n)^k, (n, 1)^k, m, \ldots)$ for some $m \neq 0$
\end{enumerate}
where $k$ is a positive integer, $\ell$ is a non-negative integer, and $m$ is an integer.
\end{lemma}

\begin{proof}
If $a_1(C) \neq 1$, then we are in case \ref{item:mgeq2}. Otherwise, $a_2$ is defined (i.e., $a_2>0$) by Lemma \ref{lem:a1=1}. If $a_2(C)=1$, then we are in case \ref{item:11}, and otherwise, let $n=a_2(C)$. Let $k$ be a positive integer such that $\ba(C)$ and $((1, n)^k, (n, 1)^k)$ agree for as many terms as possible, from the left. If there are multiple values of $k$ for which the sequences agree for the same number of terms, choose the least such $k$. Let $a_p$ be the first place where the two sequences differ. Then exactly one of the following is true:
\begin{itemize}
	\item $p=2k$, in which case $\ba(C)=((1, n)^{k-1}, 1, m, \ldots)$, where $m\neq 0$ or $n$. By Lemma \ref{lem:sequences}, if $m$ is negative, then $m<-n$. This is case \ref{item:1n1m} above.
	\item $p=2k+1$, in which case $\ba(C)=((1, n)^k, m, \ldots)$, where $m\neq 0, 1,$ or $n$. By Lemma \ref{lem:sequences2}, if $m=-1$, then the following term, $a_{p+1}$, must be strictly between $-n$ and zero. This is case \ref{item:1nm} and \ref{item:1n-1m} above.
	\item $2k+1<p\leq 4k$ and $p$ is odd, in which case $\ba(C)=((1,n)^k, (n, 1)^\ell, m, \ldots)$, where $m\neq 0$ or $n$ and $0<\ell<k$. By Lemma \ref{lem:sequences}, if $m$ is negative, then $m<-n$. This is case \ref{item:1nn1m} above.
	\item $2k+1<p \leq 4k$ and $p$ is even, in which case $\ba(C)=((1,n)^k, (n, 1)^\ell, n, m, \ldots)$, where $m \neq 0$ or $1$ and $0\leq \ell <k$. This is case \ref{item:1nn1nm} above. 
	\item $p > 4k$, in which case $\ba(C)=((1, n)^k, (n, 1)^k, m, \ldots)$, where $m \neq 0$. This is case \ref{item:1nn1c}  above. 
\end{itemize}
This concludes the proof.
\end{proof}

We now consider complexes of the forms in Lemma \ref{lem:forms}. Note that in case \ref{item:mgeq2} and \ref{item:11}, we have that $C \ll [1, n, n, 1]$ by Lemma \ref{lem:agg}. We proceed with case \ref{item:1n1m} when $m>n$. We assume throughout that $k$ is a positive integer and that $n \geq 2$.

\begin{lemma}
\label{lem:m>n}
Let $m$ and $n$ be positive integers such that $m>n$. Then 
\[ ((1, n)^k, 1, m, \dots) \gg [1, n, n, 1]. \]
\end{lemma}

\begin{proof}
Let $C_1 \in \mathfrak{C}$ be a complex such that $\ba(C_1)=((1, n)^k, 1, m, \dots)$ and let $C_2$ be such that $\ba(C_2)=[1, n, n, 1]$. We would like to show that 
\[ \varep(C_1 \otimes \ell C^*_2) =1 \]
for all $\ell>0$. Let $\{x_i\}$ be a locally simplified basis for $C_1$. See Figure \ref{subfig:somebasesa}.

We first consider the case where $k < \ell$. By \cite[Lemma 3.1]{HancockHomNewman}, the complex $\ell C_2=\ell [1, n, n, 1]$ is $\varep$-equivalent to $[(1, n)^\ell, (n, 1)^\ell]$ which has a representative with basis $\{ y_i\}_{i=0}^{4\ell}$. See Figure \ref{subfig:somebasesb}. We may dualize to obtain $\{ y^*_i\}$, a basis for $\ell C^*_2$, where $y^*_0$ generates the vertical homology, and has an outgoing horizontal arrow to $y^*_1$. See Figure \ref{subfig:somebasesc}.

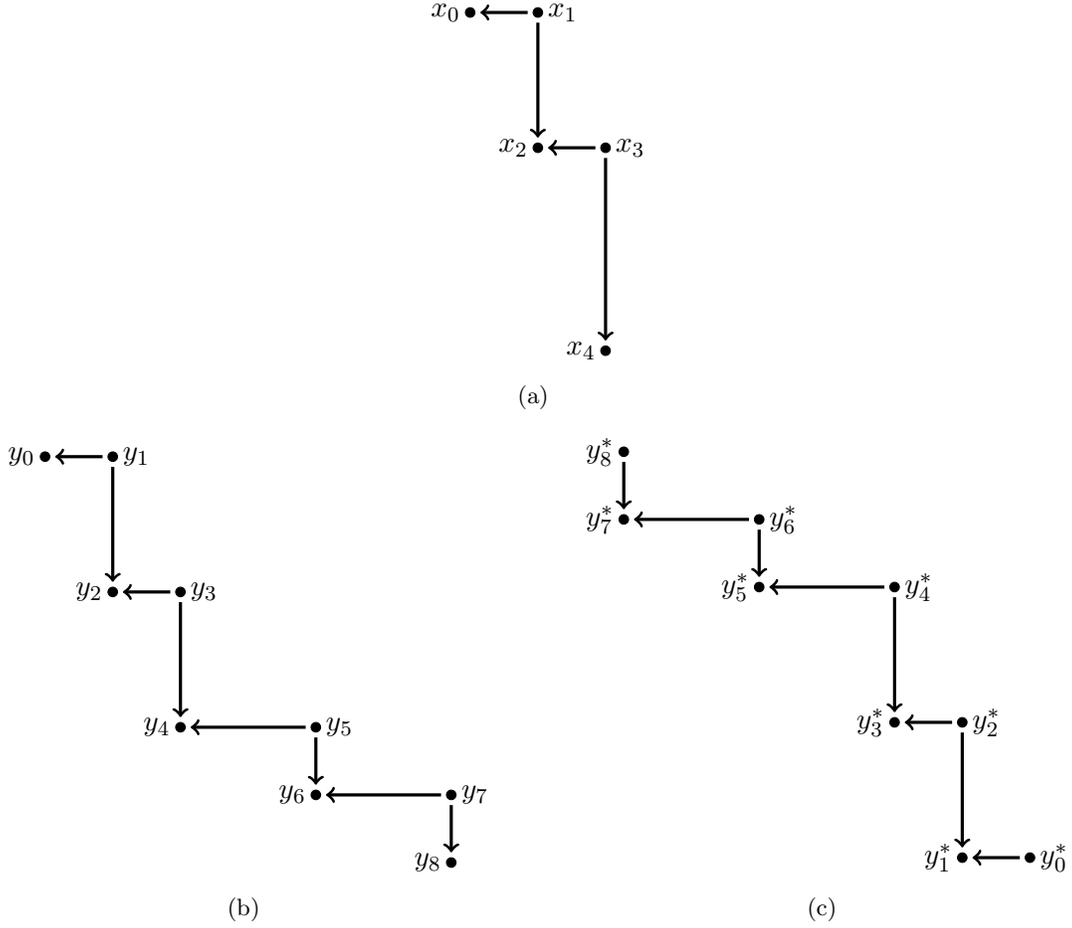
\begin{figure}[htb!]
\vspace{5pt}
\labellist
\small \hair 2pt
\endlabellist
\centering
\subfigure[]{
\begin{tikzpicture}[scale=0.9]
	\filldraw (0, 6) circle (2pt) node[] (x_0){};
	\filldraw (1, 6) circle (2pt) node[] (x_1){};
	\filldraw (1, 4) circle (2pt) node[] (x_2){};
	\filldraw (2, 4) circle (2pt) node[] (x_3){};
	\filldraw (2, 1) circle (2pt) node[] (x_4){};

	\draw [very thick, <-] (x_0) -- (x_1);
	\draw [very thick, <-] (x_2) -- (x_1);
	\draw [very thick, <-] (x_2) -- (x_3);
	\draw [very thick, <-] (x_4) -- (x_3);

	\node [left] at (x_0) {$x_0$};
	\node [right] at (x_1) {$x_1$};
	\node [left] at (x_2) {$x_2$};
	\node [right] at (x_3) {$x_3$};
	\node [left] at (x_4) {$x_4$};

\end{tikzpicture}
\label{subfig:somebasesa}
}
\\
\subfigure[]{
\begin{tikzpicture}[scale=0.9]
	\filldraw (0, 6) circle (2pt) node[] (x_0){};
	\filldraw (1, 6) circle (2pt) node[] (x_1){};
	\filldraw (1, 4) circle (2pt) node[] (x_2){};
	\filldraw (2, 4) circle (2pt) node[] (x_3){};
	\filldraw (2, 2) circle (2pt) node[] (x_4){};
	\filldraw (4, 2) circle (2pt) node[] (x_5){};
	\filldraw (4, 1) circle (2pt) node[] (x_6){};
	\filldraw (6, 1) circle (2pt) node[] (x_7){};
	\filldraw (6, 0) circle (2pt) node[] (x_8){};
	\draw [very thick, <-] (x_0) -- (x_1);
	\draw [very thick, <-] (x_2) -- (x_1);
	\draw [very thick, <-] (x_2) -- (x_3);
	\draw [very thick, <-] (x_4) -- (x_3);
	\draw [very thick, <-] (x_4) -- (x_5);
	\draw [very thick, <-] (x_6) -- (x_5);
	\draw [very thick, <-] (x_6) -- (x_7);
	\draw [very thick, <-] (x_8) -- (x_7);
	\node [left] at (x_0) {$y_0$};
	\node [right] at (x_1) {$y_1$};
	\node [left] at (x_2) {$y_2$};
	\node [right] at (x_3) {$y_3$};
	\node [left] at (x_4) {$y_4$};
	\node [right] at (x_5) {$y_5$};
	\node [left] at (x_6) {$y_6$};
	\node [right] at (x_7) {$y_7$};
	\node [left] at (x_8) {$y_8$};
\end{tikzpicture}
\label{subfig:somebasesb}
}
\hspace{15pt}
\subfigure[]{
\begin{tikzpicture}[scale=0.9]
	\filldraw (0, -6) circle (2pt) node[] (x_0){};
	\filldraw (-1, -6) circle (2pt) node[] (x_1){};
	\filldraw (-1, -4) circle (2pt) node[] (x_2){};
	\filldraw (-2, -4) circle (2pt) node[] (x_3){};
	\filldraw (-2, -2) circle (2pt) node[] (x_4){};
	\filldraw (-4, -2) circle (2pt) node[] (x_5){};
	\filldraw (-4, -1) circle (2pt) node[] (x_6){};
	\filldraw (-6, -1) circle (2pt) node[] (x_7){};
	\filldraw (-6, 0) circle (2pt) node[] (x_8){};
	\draw [very thick, ->] (x_0) -- (x_1);
	\draw [very thick, ->] (x_2) -- (x_1);
	\draw [very thick, ->] (x_2) -- (x_3);
	\draw [very thick, ->] (x_4) -- (x_3);
	\draw [very thick, ->] (x_4) -- (x_5);
	\draw [very thick, ->] (x_6) -- (x_5);
	\draw [very thick, ->] (x_6) -- (x_7);
	\draw [very thick, ->] (x_8) -- (x_7);
	\node [right] at (x_0) {$y^*_0$};
	\node [left] at (x_1) {$y^*_1$};
	\node [right] at (x_2) {$y^*_2$};
	\node [left] at (x_3) {$y^*_3$};
	\node [right] at (x_4) {$y^*_4$};
	\node [left] at (x_5) {$y^*_5$};
	\node [right] at (x_6) {$y^*_6$};
	\node [left] at (x_7) {$y^*_7$};
	\node [left] at (x_8) {$y^*_8$};
\end{tikzpicture}
\label{subfig:somebasesc}
}
\caption{Top, part of the basis $\{x_i\}$ when $k=1$; note that $x_4$ has an incoming or outgoing horizontal arrow that we have not shown. Bottom left, the basis $\{y_i\}$ when $\ell=2$. Bottom right, the dual basis $\{y^*_i\}$.}
\label{fig:somebases}
\end{figure}

We now compute $\varep(C_1 \otimes \ell C^*_2)$. We consider the complex with basis $\{ x_i y^*_j\}$, where for simplicity, we have omitted the tensor product symbol.

Let
\[ u_0 = \sum_{i=0}^{2k+1}x_iy^*_i \qquad \qquad u_1=\sum_{\substack{i=0 \\ i \ \even }}^{2k} x_{i+1} y^*_i. \]
We note the following facts:

\begin{enumerate}
	\item The chain $u_0+x_{2k+2} y^*_{2k+2}$ generates the vertical homology. Indeed, we have that $\partial^\vert u_0 = x_{2k+2}y^*_{2k+1}= \partial^\vert x_{2k+2} y^*_{2k+2}$, since
\begin{itemize}
	\item $\partial^\vert x_0y^*_0=0$
	\item $\partial^\vert x_i y^*_i=x_{i+1}y^*_i$ for $i$ odd, $1\leq i \leq 2k+1$
	\item $\partial^\vert x_i y^*_i=x_i y^*_{i-1}$ for $i$ even, $2 \leq i \leq 2k+2$
\end{itemize}
and we are working modulo $2$. 
Hence $u_0+x_{2k+2} y^*_{2k+2}$is in the kernel of $\partial^\vert$. We now show that $u_0+x_{2k+2} y^*_{2k+2}$  is not in the image of $\partial^\vert$. Consider the chain
\[ u_{-1} = \sum_{\substack{i=1 \\ i \ \odd}}^{2k+1} x_i y^*_{i+1}. \]
We have that
\[ \partial^\vert u_{-1}= u_0+x_{2k+2} y^*_{2k+2}-x_0y^*_0, \]
since
\[ \partial^\vert x_i y^*_{i+1} = x_{i+1} y^*_{i+1} + x_i y^*_i, \textup{ for }  i \ \odd, \ 1 \leq i \leq 2k+1. \]
Note that $x_0y^*_0$ is not in the image of $\partial^\vert$ and $u_0+x_{2k+2} y^*_{2k+2}-x_0y^*_0 = \partial^\vert u_{-1}$. Hence $u_0+x_{2k+2} y^*_{2k+2}$ is not in the image of $\partial^\vert$, and thus generates the vertical homology.

	\item We have that $\partial^\horz u_1=u_0$, since
\[ \partial^\horz x_{i+1}y^*_i = x_i y^*_i+x_{i+1}y^*_{i+1} \ \textup{ for } i \textup{ even}, \ \ 0 \leq i \leq 2k .\]
\end{enumerate}

With this two facts, we now consider the following generators, listed with their filtration levels, which form a summand of $C\{S(1)\}$:
\begin{align*}
	& x_iy^*_i, \quad 0\leq i \leq 2k+1  &&(0, \tau_x-\ell \tau_y)  \\
	& x_{i+1}y^*_i, \quad i \textup{ even}, 0\leq i \leq 2k &&(1, \tau_x-\ell \tau_y) \\
	& x_i y_{i+1}^*, \quad i \textup{ odd}, 1\leq i \leq 2k+1  &&(0, \tau_x-\ell\tau_y+n),
\end{align*}
where $\tau_x=\tau(C_1)$ and $\tau_y=\tau(C_2)$. We also consider the element $x_{2k+2} y^*_{2k+2}$, which is in filtration level $(0, \tau_x-\ell \tau_y+n-m)$. Note that $n-m<0$ since $n<m$.

The chain $u_0+x_{2k+2} y^*_{2k+2}$ generates the vertical homology, and the projection of this chain to $C\{S(1)\}$ is $u_0$. Moreover, $\partial^\horz u_1=u_0$. It follows that $\varep(C_1 \otimes \ell C^*_2)=1$ for $k<\ell$. See Figure \ref{subfig:m>nL}.

We now consider the case where $\ell \leq k$. Again, let $\{x_i\}$ be a locally simplified basis for $C_1$ and $\{y_i\}$ a basis for $[(1, n)^\ell, (n, 1)^\ell]$, which is $\varep$-equivalent to $\ell C_2$. We dualize $\{y_i\}$ to obtain a basis $\{y^*_i\}$, where $y^*_0$ generates the vertical homology and has an outgoing horizontal arrow to $y^*_1$.

Let
\[ u_0 = \sum_{i=0}^{2\ell} x_i y^*_i \qquad \qquad u_1= \sum_{\substack{i=0 \\ i \ \even}}^{2\ell} x_{i+1} y^*_i .\]
We note the following:
\begin{enumerate}
	\item The chain $u_0$ is in the kernel of $\partial^\vert$, since
\begin{itemize}
	\item $\partial^\vert x_0y^*_0=0$
	\item $\partial^\vert x_i y^*_i=x_{i+1}y^*_i$ for $i$ odd, $1\leq i \leq 2\ell-1$
	\item $\partial^\vert x_i y^*_i=x_i y^*_{i-1}$ for $i$ even, $2 \leq i \leq 2\ell$
\end{itemize}
and we are working modulo $2$. Furthermore, $u_0$ is not in the image of $\partial^\vert$, since $x_0y^*_0$ is not in the image of $\partial^\vert$ and 
\[ u_0-x_0y^*_0 = \partial^\vert \sum_{\substack{i=1 \\ i \ \odd}}^{2\ell-1} x_i y^*_{i+1}.\]
In particular, $u_0$ generates the vertical homology.
	\item We have that $\partial^\horz u_1 = u_0 + x_{2\ell+1}y^*_{2\ell+1}$ since
\[ \partial^\horz x_{i+1}y^*_i = x_i y^*_i + x_{i+1}y^*_{i+1} \textup{ for } i \ \even, \ 0 \leq i \leq 2\ell. \]
\end{enumerate}

Having established these two facts, we now consider the summand of $C\{S(1)\}$, generated by the following elements, listed with their filtration levels:
\begin{align*}
	& x_iy^*_i, \quad 0\leq i \leq 2k+1  &&(0, \tau_x-\ell \tau_y)  \\
	& x_{i+1}y^*_i, \quad i \textup{ even}, 0\leq i \leq 2\ell&&(1, \tau_x-\ell \tau_y) \\
	& x_i y_{i+1}^*, \quad i \textup{ odd}, 1\leq i \leq 2\ell+1  &&(0, \tau_x-\ell\tau_y+n).
\end{align*}
We also consider $x_{2\ell+1}y^*_{2\ell+1}$ in filtration level $(1-n, \tau_x-\ell \tau_y)$. Let $\partial_1$ denote the differential on $C\{S(1)\}$, and $\partial^\horz_1$ the horizontal differential on $C\{S(1)\}$. Then $ \partial^\horz_1 u_1 = u_0$, since $\partial^\horz u_1 = u_0 + x_{2\ell+1}y^*_{2\ell+1}$ and $x_{2\ell+1}y^*_{2\ell+1}$ lies outside of the subquotient complex $C\{S(1)\}$.

Since $u_0$ generates the vertical homology and $\partial^\horz_1 u_1 = u_0$, it follows that $\varepsilon(C_1 \otimes \ell C^*_2)=1$, as desired. See Figure \ref{subfig:m>nR}.
\end{proof}

\begin{figure}[htb!]
\vspace{5pt}
\labellist
\small \hair 2pt
\endlabellist
\centering
\subfigure[]{
\begin{tikzpicture}
	\filldraw[black!20!white] (-0.4, 2.6) rectangle (0.4, 5);
	\filldraw[black!20!white] (-0.4, 2.6) rectangle (1.4, 3.4);
	\shade[top color=white, bottom color=black!20!white] (-0.4, 5) rectangle (0.4, 6.25);
	\begin{scope}[thin, gray]
		\draw [<->] (-1, 0) -- (3, 0);
		\draw [<->] (0, -1) -- (0, 6);
	\end{scope}
	\filldraw (0.3, 3.3) circle (2pt) node[] (a){};
	\filldraw (0.1, 3.1) circle (2pt) node[] (c){};
	\filldraw (-0.1, 2.9) circle (2pt) node[] (e){};
	\filldraw (-0.3, 2.7) circle (2pt) node[] (g){};
	
	\filldraw (1.3, 3.2) circle (2pt) node[] (b){};
	\filldraw (1.3, 2.8) circle (2pt) node[] (f){};
	
	\filldraw (0, 5) circle (2pt) node[] (d){};
	\filldraw (-0.3, 5) circle (2pt) node[] (h){};
	
	\filldraw (0, 1.7) circle (2pt) node[] (i){};
	
	\draw [very thick, <-] (a) -- (b);
	\draw [very thick, <-] (c) -- (b);
	\draw [very thick, <-] (c) -- (d);
	\draw [very thick, <-] (e) -- (d);
	\draw [very thick, <-] (e) -- (f);
	\draw [very thick, <-] (g) -- (f);
	\draw [very thick, <-] (g) -- (h);
	\draw [very thick, <-] (i) .. controls (-0.7, 2.6) and (-0.7, 4) .. (h);
	
	\node [above right] at (a) {$x_0y^*_0$};
	\node [right] at (b) {$x_1y^*_0$};


	\node [above right] at (d) {$x_1y^*_2$};
	
	\node [right] at (f) {$x_3y^*_2$};
	\node [below right] at (g) {$x_3y^*_3$};
	\node [above left] at (h) {$x_3y^*_4$};

	\node [below] at (i) {$x_4y^*_4$};

\end{tikzpicture}
\label{subfig:}
}
\hspace{25pt}
\subfigure[]{
\begin{tikzpicture}
	\filldraw[black!20!white] (-0.2, 2.8) rectangle (0.2, 5);
	\filldraw[black!20!white] (-0.2, 2.8) rectangle (1.2, 3.2);
	\shade[top color=white, bottom color=black!20!white] (-0.2, 5) rectangle (0.2, 6.25);
	\begin{scope}[thin, gray]
		\draw [<->] (-1, 0) -- (3, 0);
		\draw [<->] (0, -1) -- (0, 6);
	\end{scope}
	\filldraw (0.1, 3.15) circle (2pt) node[] (x_0y_0){};
	\filldraw (-0.1, 2.85) circle (2pt) node[] (u'_0){};
	\filldraw (0, 2) circle (2pt) node[] (x_y_2k+2){};v
	\filldraw (1, 3) circle (2pt) node[] (u_1){};
	\filldraw (-0.1, 5) circle (2pt) node[] (u_-1){};
	\draw [very thick, <-] (x_0y_0) -- (u_1);
	\draw [very thick, <-] (u'_0) -- (u_1);
	\draw [very thick, <-] (u'_0) -- (u_-1);
	\draw [very thick, <-] (x_y_2k+2) .. controls (-0.4, 3) and (-0.4, 4) .. (u_-1);
	\node [above right, xshift=2pt] at (x_0y_0) {$x_0y^*_0$};
	\node [right] at (u_1) {$u_1$};
	\node [below right] at (u'_0) {$u_0-x_0y^*_0$};
	\node [left] at (x_y_2k+2) {$x_{2k+2}y^*_{2k+2}$};
	\node [above] at (u_-1) {$u_{-1}$};	
	
	\draw [decorate,decoration={brace,amplitude=10pt,mirror}] (-0.4, 5) -- (-0.4, 2) node [midway, left, xshift=-8pt] {$m$};
	\draw [decorate,decoration={brace,amplitude=6pt}] (0.15, 5) -- (0.15, 3.2) node [midway, right, xshift=4pt] {$n$};
\end{tikzpicture}
\label{subfig:m>nL}
}
\caption{The summand involved in the proof of Lemma \ref{lem:m>n} when $k<\ell$. Left, when $k=1$ and $\ell=2$. The two unlabeled generators are $x_1y^*_1$ and $x_2y^*_2$. Right, the general case. We include the arrow to $x_{2k+2}y^*_{2k+2}$ out of subquotient $C\{S(1)\}$ to illustrate the roles of the relative lengths of $m$ and $n$.}
\label{fig:outofsubquotient}
\end{figure}
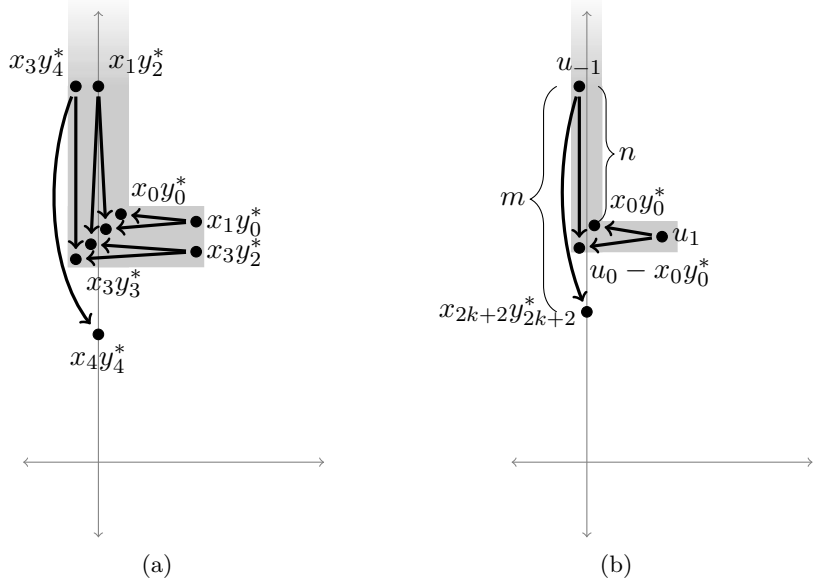

\begin{figure}[htb!]
\vspace{5pt}
\labellist
\small \hair 2pt
\endlabellist
\centering
\subfigure[]{
\begin{tikzpicture}
	\filldraw[black!20!white] (-0.3, 2.7) rectangle (0.3, 5);
	\filldraw[black!20!white] (-0.3, 2.7) rectangle (1.3, 3.3);
	\shade[top color=white, bottom color=black!20!white] (-0.3, 5) rectangle (0.3, 6.25);
	\begin{scope}[thin, gray]
		\draw [<->] (-3, 0) -- (2, 0);
		\draw [<->] (0, -1) -- (0, 6.6);
	\end{scope}
	\filldraw (0.2, 3.2) circle (2pt) node[] (a){};
	\filldraw (0, 3) circle (2pt) node[] (c){};
	\filldraw (-0.2, 2.8) circle (2pt) node[] (e){};
	
	\filldraw (1.2, 3.1) circle (2pt) node[] (b){};
	\filldraw (1.2, 2.8) circle (2pt) node[] (f){};
	
	\filldraw (-0.1, 6) circle (2pt) node[] (d){};
	
	\filldraw (-2.3, 2.8) circle (2pt) node[] (g){};
	
	\draw [very thick, <-] (a) -- (b);
	\draw [very thick, <-] (c) -- (b);
	\draw [very thick, <-] (c) -- (d);
	\draw [very thick, <-] (e) -- (d);
	\draw [very thick, <-] (e) -- (f);
	\draw [very thick, <-] (g) .. controls (-1, 2.5) and (0, 2.5) .. (f);
	
	\node [above right] at (a) {$x_0y^*_0$};
	\node [above right] at (b) {$x_1y^*_0$};

	\node [above right] at (d) {$x_1y^*_2$};
	
	\node [above left, yshift=2pt] at (c) {$x_1y^*_1$};	
	\node [above left, yshift=-3pt] at (e) {$x_2y^*_2$};
	
	\node [below right] at (f) {$x_3y^*_2$};
	\node [below left] at (g) {$x_3y^*_3$};

\end{tikzpicture}
\label{subfig:}
}
\hspace{25pt}
\subfigure[]{
\begin{tikzpicture}
		\filldraw[black!20!white] (-0.2, 2.8) rectangle (0.2, 5);
	\filldraw[black!20!white] (-0.2, 2.8) rectangle (1.2, 3.2);
	\shade[top color=white, bottom color=black!20!white] (-0.2, 5) rectangle (0.2, 6.25);
	\begin{scope}[thin, gray]
		\draw [<->] (-3, 0) -- (2, 0);
		\draw [<->] (0, -1) -- (0, 6.6);
	\end{scope}
	\filldraw (0, 3) circle (2pt) node[] (u_0){};
	\filldraw (-2, 3) circle (2pt) node[] (x_2l+1y_2l+1){};
	\filldraw (1, 3) circle (2pt) node[] (u_1){};
	\draw [very thick, <-] (u_0) .. controls (.4, 3.1) and (0.6, 3.1) .. (u_1);
	\draw [very thick, <-] (x_2l+1y_2l+1) .. controls (-1, 2.8) and (0, 2.8) .. (u_1);
	\node [above] at (x_2l+1y_2l+1) {$x_{2\ell+1}y^*_{2\ell+1}$};
	\node [right] at (u_1) {$u_1$};
	\node [above] at (u_0) {$u_0$};
	
	\draw [decorate,decoration={brace,amplitude=10pt,mirror}] (-2, 2.75) -- (1, 2.75) node [midway, below, yshift=-8pt] {$n$};
(3.5,0.65) -- (3.5,6.5)
\end{tikzpicture}
\label{subfig:m>nR}
}
\caption{The chains involved in the proof of Lemma \ref{lem:m>n} when $\ell \leq k$. Left, when $k=\ell=1$. Right, the general case. We include the arrow to $x_{2\ell+1}y^*_{2\ell+1}$ out of the subquotient $C\{S(1)\}$ for emphasis.}
\label{fig:outofsubquotient2}
\end{figure}
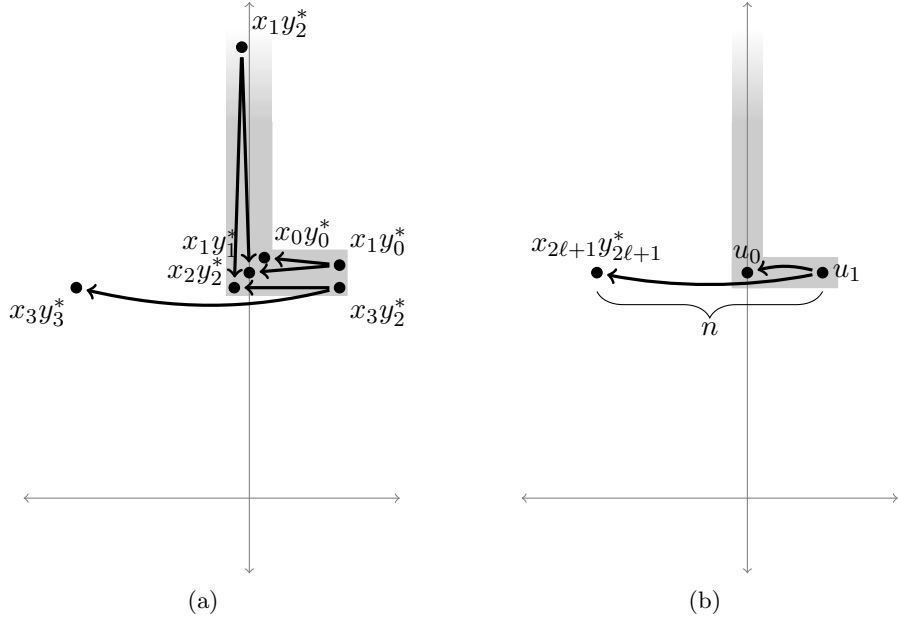

The complexes shown in Figures \ref{fig:outofsubquotient} and \ref{fig:outofsubquotient2} model two types of behavior that we will see again in the tensor products in this section. In Figure \ref{subfig:m>nL}, we identify a summand of the subquotient complex $C\{S(1)\}$. A key ingredient is the arrow from $u_{-1}$ to $x_{2k+2}y^*_{2k+2}$ which leaves the subquotient complex $C\{S(1)\}$, which allows us to find a generator of the vertical homology that is in the image of the horizontal differential on $C\{S(1)\}$. Similarly, in Figure \ref{subfig:m>nR}, we have an arrow from $u_1$ which leaves the complex $C\{S(1)\}$, allowing us to conclude that $u_0$ is in the image of the horizontal differential on $C\{S(1)\}$.

\begin{lemma}
\label{lem:mn}
Let $m$ and $n$ be positive integers such that $m<n$. Then
\[ ((1, n)^k, 1, m, \dots) = [(1, n)^k, (n, 1)^k] + (1, m, \ldots). \]
\end{lemma}

\begin{proof}
Consider bases $\{ x_i \}$ and $\{ y_i \}$ for $((1, n)^k, 1, m, \dots)$ and $[(1, n)^k, (n, 1)^k]$ respectively. We may dualize $\{ y_i \}$ to obtain $\{ y_i^* \}$, where $y_0^*$ generates the vertical homology and has an outgoing horizontal arrow to $y_1^*$.

Let 
\[ u_0 = \sum_{i=0}^{2k} x_i y_i^* \qquad \qquad u_1= \sum_{\substack{i=1 \\ i \ \odd}}^{2k+1} x_i y_{i-1}^*. \]
Observe the following:
\begin{enumerate}
	\item The chain $u_0$ generates the vertical homology. Indeed, we have that $\partial^\vert u_0 =0$, since 
	\begin{itemize}
		\item $\partial^\vert x_0 y_0^*=0$
		\item $\partial^\vert x_i y_i^* = x_{i+1}y_i^*$ for $i$ odd, $1\leq i \leq 2k-1$
		\item $\partial^\vert x_i y_i^* = x_i y_{i-1}^*$ for $i$ even, $2 \leq i \leq 2k$
	\end{itemize}
	and we are working modulo $2$. Moreover, note that $x_0 y_0^*$ is not in the image of $\partial^\vert$ and that 
	\begin{align*}
		u_0 - x_0 y_0^* &= \sum_{i=1}^{2k} x_i y_i^* \\
		&= \partial^\vert \sum_{\substack{i=1 \\ i \ \odd }}^{2k-1} x_i y_{i+1}^*,
	\end{align*}
	so it follows that $u_0$ is not in the image of $\partial^\vert$. Hence $u_0$ generates the vertical homology, as desired.
	\item We have that $\partial^\horz u_1 = u_0 + x_{2k+1}y_{2k+1}^*$, since
		\[ \partial^\horz x_iy^*_{i-1}=x_{i-1}y^*_{i-1}+x_iy^*_i, \ \textup{for } i \textup{ odd}, \ 1\leq i \leq 2k+1. \]
	\item We have that
		\[ \partial^\vert u_1 = \sum_{i=0}^{2k} x_{i+2} y_i^*, \]
	since 
		\[ \partial^\vert x_iy^*_{i-1} = x_{i+1}y^*_{i-1}, \ \textup{for } i \textup{ odd}, \ 1\leq i \leq 2k+1. \]
\end{enumerate}

Using these three facts, we may consider the direct summand of $C\{S(1,m)\}$ generated by the following generators in the filtration levels listed:
\begin{align*}
	& x_iy^*_i, \quad 0\leq i \leq 2k  &&(0, \tau_x-\tau_y)  \\
	& x_iy^*_{i-1}, \quad i \textup{ odd}, 1\leq i \leq 2k+1 &&(1, \tau_x-\tau_y) \\
	& x_{2k+2}y^*_{2k} && (1, \tau_x-\tau_y-m) \\
	& x_i y_{i+1}^*, \quad i \textup{ odd}, 1\leq i \leq 2k-1  &&(0, \tau_x-\tau_y+n).
\end{align*}

\noindent Note that $x_{2k+1} y^*_{2k+1}$ has filtration level $(1-n, \tau_x-\tau_y)$ and  $x_{i+2}y_i^*$ has filtration level $(1, \tau_x-\tau_y -n)$ for $0 \leq i \leq 2k-1$. In particular, these generators fall outside of $C\{S(1, m)\}$.
Thus, we have that $\partial_{1,m} u_1 = u_0+x_{2k+2}y_{2k} $, where $\partial_{1,m}$ denotes the induced differential on $C\{S(1,m)\}$. More specifically,
\[ \partial_{1,m}^\horz u_1 = u_0 \qquad \partial_{1,m}^\vert u_0 = x_{2k+2}y_{2k}. \]
See Figure \ref{fig:lemmn}.
Thus, it follows that $a_1$ of the tensor product is $1$ and $a_2$ of the tensor product is $m$, as desired.
\end{proof}

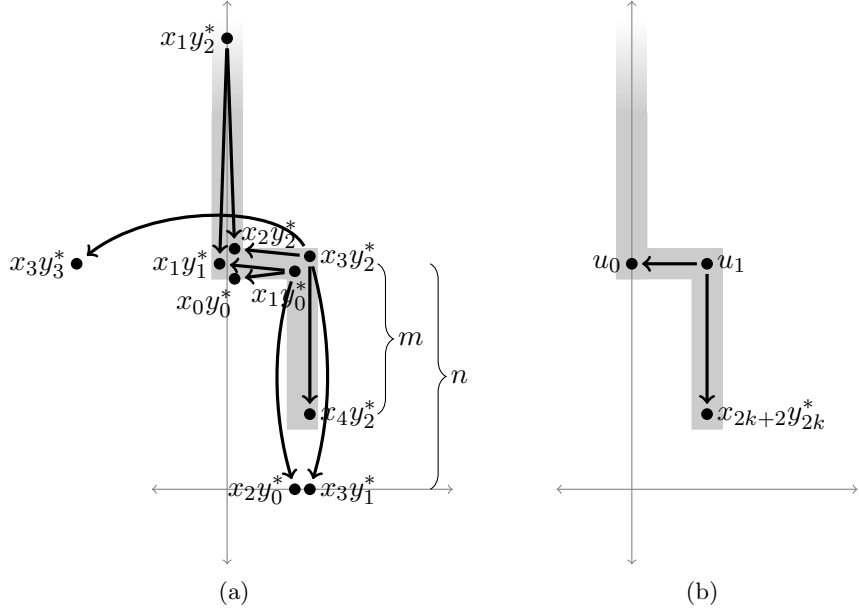
\begin{figure}[htb!]
\vspace{5pt}
\labellist
\small \hair 2pt
\endlabellist
\centering
\subfigure[]{
\begin{tikzpicture}
	\filldraw[black!20!white] (-0.2, 2.8) rectangle (0.2, 5);
	\filldraw[black!20!white] (-0.2, 2.8) rectangle (1.2, 3.2);
	\shade[top color=white, bottom color=black!20!white] (-0.2, 5) rectangle (0.2, 6.25);
	\filldraw[black!20!white] (0.8, 0.8) rectangle (1.2, 3.2);
	\begin{scope}[thin, gray]
		\draw [<->] (-1, 0) -- (3, 0);
		\draw [<->] (0, -1) -- (0, 6.5);
	\end{scope}
	\filldraw (0.1, 2.8) circle (2pt) node[] (x_0y_0){};
	\filldraw (-0.1, 3) circle (2pt) node[] (x_1y_1){};
	\filldraw (0.1, 3.2) circle (2pt) node[] (x_2y_2){};
	\filldraw (0.9, 2.9) circle (2pt) node[] (x_1y_0){};
	\filldraw (1.1, 3.1) circle (2pt) node[] (x_3y_2){};
	\filldraw (1.1, 1) circle (2pt) node[] (x_4y_2){};
	\filldraw (0, 6) circle (2pt) node[] (x_1y_2){};
	
	\filldraw (0.9, 0) circle (2pt) node[] (x_2y_0){};
	\filldraw (1.1, 0) circle (2pt) node[] (x_3y_1){};
	
	\filldraw (-2, 3) circle (2pt) node[] (x_3y_3){};
	
	\draw [very thick, <-] (x_0y_0) -- (x_1y_0);
	\draw [very thick, <-] (x_1y_1) -- (x_1y_0);
	\draw [very thick, <-] (x_1y_1) -- (x_1y_2);
	\draw [very thick, <-] (x_2y_2) -- (x_1y_2);
	\draw [very thick, <-] (x_2y_2) -- (x_3y_2);
	\draw [very thick, <-] (x_4y_2) -- (x_3y_2);
	
	\draw [very thick, <-] (x_2y_0) .. controls (0.6, 1) and (0.6, 2) .. (x_1y_0);
	\draw [very thick, <-] (x_3y_1) .. controls (1.4, 1) and (1.4, 2) .. (x_3y_2);
	
	\draw [very thick, <-] (x_3y_3) .. controls (-1, 3.8) and (0.7, 3.8) .. (x_3y_2);
	
	\node [below] at (-0.3, 2.8) {$x_0y^*_0$};
	\node [below] at (0.7, 2.9) {$x_1y^*_0$};
	\node [left] at (x_1y_1) {$x_1y^*_1$};
	\node [left] at (x_1y_2) {$x_1y^*_2$};
	\node [right] at (0.05, 3.4) {$x_2y^*_2$};
	\node [right] at (x_3y_2) {$x_3y^*_2$};
	\node [right] at (x_4y_2) {$x_4y^*_2$};
	
	\node [left] at (x_2y_0) {$x_2y^*_0$};
	\node [right] at (x_3y_1) {$x_3y^*_1$};
	
	\node [left] at (x_3y_3) {$x_3y^*_3$};
	
	\draw [decorate,decoration={brace,amplitude=6pt,mirror}] (2, 1) -- (2, 3) node [midway, right, xshift=4pt] {$m$};
	\draw [decorate,decoration={brace,amplitude=6pt,mirror}] (2.7, 0) -- (2.7, 3) node [midway, right, xshift=4pt] {$n$};	
\end{tikzpicture}
\label{subfig:}
}
\hspace{15pt}
\subfigure[]{
\begin{tikzpicture}
	\filldraw[black!20!white] (-0.2, 2.8) rectangle (0.2, 5);
	\filldraw[black!20!white] (-0.2, 2.8) rectangle (1.2, 3.2);
	\shade[top color=white, bottom color=black!20!white] (-0.2, 5) rectangle (0.2, 6.25);
	\filldraw[black!20!white] (0.8, 0.8) rectangle (1.2, 3.2);
	\begin{scope}[thin, gray]
		\draw [<->] (-1, 0) -- (3, 0);
		\draw [<->] (0, -1) -- (0, 6.5);
	\end{scope}
	
	\filldraw (0, 3) circle (2pt) node[] (u_0){};
	\filldraw (1, 3) circle (2pt) node[] (u_1){};
	\filldraw (1, 1) circle (2pt) node[] (xy){};
	
	\draw [very thick, <-] (u_0) -- (u_1);
	\draw [very thick, <-] (xy) -- (u_1);
	
	\node [left] at (u_0) {$u_0$};
	\node [right] at (u_1) {$u_1$};
	\node [right] at (xy) {$x_{2k+2}y^*_{2k}$};
	
\end{tikzpicture}
}
\caption{The summand in the proof of Lemma \ref{lem:mn}. Left, the case where $k=1$. Right, the general case.}
\label{fig:lemmn}
\end{figure}


The next several lemmas 
are similar to Lemmas \ref{lem:m>n} and \ref{lem:mn}, but concern the other cases of Lemma \ref{lem:forms}. More specifically, Lemma \ref{lem:mn0} completes case \ref{item:1n1m}, Lemmas \ref{lem:0nm}-\ref{lem:-m2} deal with case \ref{item:1nm}, Lemma \ref{lem:1n-1m} with case \ref{item:1n-1m}, Lemmas \ref{lem:1nn1m}-\ref{lem:1nn1mb} with case \ref{item:1nn1m}, Lemmas \ref{lem:1nnnm}-\ref{lem:1nnnmb} with case \ref{item:1nn1nm}, and Lemmas \ref{lem:1nkn1k}-\ref{lem:last} with \ref{item:1nn1c}.
Since the proofs of the lemmas are all similar, we advise the reader to skip to the proof of Proposition \ref{prop:ArchA} at the end of this section.

\begin{lemma}
\label{lem:mn0}
Let $m$ and $n$ be positve integers with $m > n$. Then 
\[ ((1, n)^k, 1, -m, \ldots) \gg [1, n, n, 1]. \]
\end{lemma}


\begin{proof}
Since the proofs are similar, we adopt the same notation used in the proof of Lemma \ref{lem:m>n}; namely, a locally simplified basis $\{x_i\}$ for the complex $C_1=((1, n)^k, 1, -m, \ldots)$ and a simplified basis $\{y_i\}$ for $[(1, n)^\ell, (n, 1)^\ell]$, which is $\varep$-equivalent to $\ell C_2 = \ell[1, n, n, 1]$. We would like to show that
\[ \varepsilon(C_1 \otimes \ell C^*_2) =1 \]
for all $\ell >0$.

We first consider the case $k < \ell$. Consider 
\[ u_0 = \sum_{i=0}^{2k+1} x_i y^*_i \qquad \qquad u_1 =\sum_{\substack{i=0 \\ i \ \even}}^{2k} x_{i+1}y^*_i,\]
and note that the chain $u_0$ is in the kernel of $\partial^\vert$.
Moreover, the chain $u_0$ is not in the image of $\partial^\vert$, since
\[ \partial^\vert \sum_{\substack{i=1 \\ i \ \odd}}^{2k+1} x_i y^*_{i+1} = u_0-x_0y^*_0, \]
and $x_0y^*_0$ is not in the image of $\partial^\vert$. Thus $u_0$ generates the vertical homology.
We also have that
\[ \partial^\horz u_1 = u_0. \]
We consider the summand of $C\{S(1)\}$ generated by
\begin{align*}
	& x_iy^*_i, \quad 0\leq i \leq 2k+1  &&(0, \tau_x-\ell \tau_y)  \\
	& x_{i+1}y^*_i, \quad i \textup{ even}, 0\leq i \leq 2k &&(1, \tau_x-\ell \tau_y) \\
	& x_i y_{i+1}^*, \quad i \textup{ odd}, 1\leq i \leq 2k+1  &&(0, \tau_x-\ell\tau_y+n),
\end{align*}
where $\tau_x=\tau(C_1)$ and $\tau_y=\tau(C_2)$. Since $u_0$, a generator for the vertical homology, is in the image of the horizontal boundary on $C\{S(1)\}$, it follows that $\varep(C_1 \otimes \ell C^*_2)=1$.

The case $\ell \leq k$ follows similarly; namely, let 
\[ u_0 = \sum_{i=0}^{2\ell} x_i y^*_i \qquad \qquad u_1 =\sum_{\substack{i=0 \\ i \ \even}}^{2\ell} x_{i+1}y^*_i.\]
Then $u_0$ generates the vertical homology and is equal to $\partial_1^\horz u_1$, where $\partial_1$ denotes the induced differential on $C\{S(1)\}$. Hence, $\varep(C_1 \otimes C^*_2)=1$.
\end{proof}

This concludes case \ref{item:1n1m} from Lemma \ref{lem:forms}. The next three lemma concern case \ref{item:1nm} from Lemma \ref{lem:forms}.

\begin{lemma} \label{lem:0nm}
Let $m$ and $n$ be positive integers such that $n<m$. Then 
\[ ((1, n)^k, m, \ldots) = [(1, n)^k, (n, 1)^k] - (n, \ldots). \]
\end{lemma}

\begin{proof} 
Let $C_1$ be a complex with $\ba(C_1)=[(1, n)^k, (n, 1)^k]$ and $C_2$ a complex with $\ba(C_2)=((1, n)^k, m, \ldots)$. 
Let $\{x_i\}$ and $\{y_i\}$ be locally simplified bases for $C_1$ and $C_2$ respectively. Consider $C_1 \otimes C^*_2$. Let
\[ u_0 = \sum_{i=0}^{2k} x_i y^*_i \qquad \qquad u_1 = \sum_{\substack{i=1 \\ i \ \odd}}^{2k+1} x_i y^*_{i-1}. \]
The elements
\begin{align*}
	& x_iy^*_i, \quad 0\leq i \leq 2k  &&(0, \tau_x-\tau_y)  \\
	& x_iy^*_{i-1}, \quad i \textup{ odd}, 1\leq i \leq 2k-1 &&(1, \tau_x-\tau_y) \\
	& x_{2k+1}y^*_{2k}  &&(n, \tau_x-\tau_y) \\
	& x_i y_{i+1}^*, \quad i \textup{ odd}, 1\leq i \leq 2k-1  &&(0, \tau_x-\tau_y+n).
\end{align*}
generate a summand of $C\{S(n)\}$, where $\tau_x = \tau(C_1)$ and $\tau_y=\tau(C_2)$.
The chain $u_0$ generates the vertical homology and $\partial_n^\horz u_1 = u_0$, where $\partial_n$ denotes the induced differential on $C\{S(n)\}$.  Moreover, the filtration level of $u_0$ is $(0, \tau_x -\tau_y)$ and of $u_1$ is $(n, \tau_x-\tau_y)$. Hence $a_1(C_1 \otimes C^*_2)=n$.
\end{proof}

\begin{lemma} \label{lem:fortorusknots}
Let $m$ and $n$ be positive integers such that $m<n$. Then 
\[ ((1, n)^k, m, \ldots) = [(1, n)^k, (n, 1)^k] + (m, \ldots). \]
\end{lemma}

\begin{proof}
Let $C_1$ be a complex with $\ba(C_1)=((1, n)^k, m, \ldots)$ and $C_2$ a complex with $\ba(C_2)=[(1, n)^k, (n, 1)^k]$. Let $\{x_i\}$ and $\{y_i\}$ be locally simplified bases for $C_1$ and $C_2$ respectively. Consider $C_1 \otimes C^*_2$. Let
\[ u_0 = \sum_{i=0}^{2k} x_i y^*_i \qquad \qquad u_1 = \sum_{\substack{i=1 \\ i \ \odd}}^{2k+1} x_i y^*_{i-1}. \]
The elements
\begin{align*}
	& x_iy^*_i, \quad 0\leq i \leq 2k  &&(0, \tau_x-\tau_y)  \\
	& x_iy^*_{i-1}, \quad i \textup{ odd}, 1\leq i \leq 2k-1 &&(1, \tau_x-\tau_y) \\
	& x_{2k+1}y^*_{2k}  &&(m, \tau_x-\tau_y) \\
	& x_i y_{i+1}^*, \quad i \textup{ odd}, 1\leq i \leq 2k-1  &&(0, \tau_x-\tau_y+n).
\end{align*}
form a direct summand of $C\{S(m)\}$ and $\partial^\horz_m u_1 = u_0$, where $\partial_m$ is the differential on $C\{S(m)\}$. The chain $u_0$ generates the vertical homology. The filtration level of $u_0$ is $(0, \tau_x-\tau_y)$ and that of $u_1$ is $(m, \tau_x-\tau_y)$. Thus, $a_1(C_1 \otimes C^*_2)=m$.
\end{proof}

\begin{lemma} \label{lem:-m2}
Let $m$ and $n$ be positive integers with $m>1$. Then 
\[ ((1, n)^k, -m, \ldots) = [(1, n)^k, (n, 1)^k] - (\textup{min}\{m, n\}, \ldots). \]
\end{lemma}

\begin{proof}
Let $C_1$ be a complex with $\ba(C_1)=[(1, n)^k, (n, 1)^k]$ and $C_2$ a complex with $\ba(C_2) = ((1, n)^k, -m, \ldots)$. Let $\{x_i\}$ and $\{y_i\}$ be a locally simplified bases for $C_1$ and $C_2$ respectively. Consider $C_1 \otimes C^*_2$.

Suppose $n<m$. Then let
\[ u_0 = \sum_{i=0}^{2k} x_i y^*_i \qquad \qquad u_1 = \sum_{\substack{i=1 \\ i \ \odd}}^{2k+1} x_i y^*_{i-1}. \]
The elements
\begin{align*}
	& x_iy^*_i, \quad 0\leq i \leq 2k  &&(0, \tau_x-\tau_y)  \\
	& x_iy^*_{i-1}, \quad i \textup{ odd}, 1\leq i \leq 2k-1 &&(1, \tau_x-\tau_y) \\
	& x_{2k+1}y^*_{2k}  &&(n, \tau_x-\tau_y) \\
	& x_i y_{i+1}^*, \quad i \textup{ odd}, 1\leq i \leq 2k-1  &&(0, \tau_x-\tau_y+n).
\end{align*}
form a direct summand of $C\{S(n)\}$, where $\tau_x = \tau(C_1)$ and $\tau_y=\tau(C_2)$. The chain $u_0$ generates the vertical homology and $\partial^\horz u_1 = u_0$. The filtration level of $u_0$ is $(0, \tau_x-\tau_y)$ and that of $u_1$ is $(n, \tau_x-\tau_y)$. Thus, $a_1(C_1 \otimes C^*_2)=n$.

Now suppose that $n>m$. Then let
\[ u_0 = \sum_{i=0}^{2k} x_i y^*_i \qquad \qquad u_1 = x_{2k}y_{2k+1} + \sum_{\substack{i=1 \\ i \ \odd}}^{2k-1} x_i y^*_{i-1}. \]
The elements
\begin{align*}
	& x_iy^*_i, \quad 0\leq i \leq 2k  &&(0, \tau_x-\tau_y)  \\
	& x_iy^*_{i-1}, \quad i \textup{ odd}, 1\leq i \leq 2k-1 &&(1, \tau_x-\tau_y) \\
	& x_{2k}y^*_{2k+1}  &&(m, \tau_x-\tau_y) \\
	& x_i y_{i+1}^*, \quad i \textup{ odd}, 1\leq i \leq 2k-1  &&(0, \tau_x-\tau_y+n).
\end{align*}
form a direct summand of $C\{S(m)\}$, where $\tau_x = \tau(C_1)$ and $\tau_y=\tau(C_2)$. The chain $u_0$ generates the vertical homology and $\partial^\horz u_1 = u_0$. The filtration level of $u_0$ is $(0, \tau_x-\tau_y)$ and that of $u_1$ is $(m, \tau_x-\tau_y)$. Thus, $a_1(C_1 \otimes C^*_2)=m$.
\end{proof}

\noindent We now consider case \ref{item:1n-1m} from Lemma \ref{lem:forms}.

\begin{lemma}\label{lem:1n-1m}
Let $m$ and $n$ be positive integers such that $m<n$. Then 
\[ ((1, n)^k, -1, -m, \ldots)=[(1, n)^k, (n, 1)^k] - (1, m, \ldots). \]
\end{lemma}

\begin{proof}
Let $C_1$ be of the form $[(1, n)^k, (n, 1)^k]$ with simplified basis $\{x_i\}$ and $C_2$ be of the form $((1, n)^k, -1, -m, \ldots)$ with locally simplified basis $\{y_i\}$. We will show that $C_1 \otimes C^*_2$ is of the form $(1, m, \ldots)$.

Let
\[ u_0 = \sum_{i=0}^{2k} x_i y^*_i \qquad \qquad u_1= x_{2k}y^*_{2k+1} + \sum_{\substack{i=1 \\ i \ \odd}}^{2k-1} x_i y^*_{i-1}. \]



The following generators, listed below with their filtration levels, form a direct summand of $C\{S(1, m)\}$:
\begin{align*}
	& x_iy^*_i, \quad 0\leq i \leq 2k  &&(0, \tau_x-\tau_y)  \\
	& x_i y^*_{i-1}, \quad i \textup{ odd}, 1 \leq i \leq 2k-1 &&(1, \tau_x- \tau_y) \\
	& x_{2k} y^*_{2k+1},  &&(1, \tau_x- \tau_y) \\
	& x_{2k} y^*_{2k+2},  &&(1, \tau_x-\tau_y-m) \\
	& x_i,y^*_{i+1}, \quad i \textup{ odd}, 1 \leq i \leq 2k-1, &&(0, \tau_x-\tau_y+n).
\end{align*}
Moreover, the chain $u_0$ generates the vertical homology and
\[ \partial_{1,m} u_1 = u_0 + x_{2k} y^*_{2k+2}, \]
where $\partial_{1, m}$ denotes the induced differential on $C\{S(1,m)\}$, $u_0$ is in filtration level $(0, \tau_x-\tau_y)$, $u_1$ in $(1, \tau_x-\tau_y)$, and $x_{2k} y^*_{2k+2}$ in $(0, \tau_x-\tau_y-m)$.
Thus,
\[ a_1(C_1 \otimes C^*_2)=1 \qquad \textup{and} \qquad a_2(C_1\otimes C^*_2)=m, \]
completing the proof.
\end{proof}

The following three lemmas concern case \ref{item:1nn1m} from Lemma \ref{lem:forms}. 

\begin{lemma} \label{lem:1nn1m}
Let $\ell$, $m$, and $n$ be positive integers. 
Suppose $\ell < k$ and $m<n$. Then 
\[ ((1, n)^k, (n, 1)^\ell, m, \ldots) = [(1, n)^k, (n, 1)^k] + (n, \ldots). \]
\end{lemma}

\begin{proof}
Let $C_1$ be a complex such that $\ba(C_1) = ((1, n)^k, (n, 1)^\ell, m, \ldots)$ and $C_2$ such that $\ba(C_2)=[(1, n)^k, (n, 1)^k]$. Let $\{x_i\}$ and $\{y_i\}$ be locally simplified bases for $C_1$ and $C_2$ respectively. Consider $C_1 \otimes C^*_2$. Let
\[ u_0 = \sum_{i=0}^{2k + 2\ell} x_i y^*_i \qquad \qquad u_1 = \sum_{\substack{i=1 \\ i \ \odd}}^{2k+2\ell+1} x_i y^*_{i-1}. \]
The elements
\begin{align*}
	& x_iy^*_i, \quad 0\leq i \leq 2k+2\ell  &&(0, \tau_x-\tau_y)  \\
	& x_iy^*_{i-1}, \quad i \textup{ odd}, 1\leq i \leq 2k-1 &&(1, \tau_x-\tau_y) \\
	& x_iy^*_{i-1}, \quad i \textup{ odd}, 2k+1 \leq i \leq 2k+2\ell-1 &&(n, \tau_x-\tau_y) \\
	& x_{2k+2\ell+1}y^*_{2k+2\ell},  &&(m, \tau_x-\tau_y) \\
	& x_i y_{i+1}^*, \quad i \textup{ odd}, 1\leq i \leq 2k-1  &&(0, \tau_x-\tau_y+n) \\
	& x_i y_{i+1}^*, \quad i \textup{ odd}, 2k+1 \leq i \leq 2k+2\ell-1  &&(0, \tau_x-\tau_y+1)
\end{align*}
form a direct summand of $C\{S(n)\}$ such that $\partial_n^\horz u_1 = u_0$ (where $\partial_n$ denotes the induced differential on $C\{S(n)\}$) and $u_0$ is a generator of the vertical homology. The filtration level of $u_0$ is $(0, \tau_x-\tau_y)$ and of $u_1$ is $(n, \tau_x-\tau_y)$. Thus $a_1(C_1 \otimes C^*_2)=n$.
\end{proof}

\begin{lemma}
Let $\ell$, $m$, and $n$ be positive integers. 
Suppose $\ell < k$ and $m>n$. Then 
\[ ((1, n)^k, (n, 1)^\ell, m, \ldots) = [(1, n)^k, (n, 1)^k] - (n, \ldots). \]
\end{lemma}

\begin{proof}
We proceed as in previous lemmas. Let $C_1$ be a complex such that $\ba(C_1)=[(1, n)^k, (n, 1)^k] $ and $C_2$ such that $\ba(C_2)=((1, n)^k, (n, 1)^\ell, m, \ldots)$. Let $\{x_i\}$ and $\{y_i\}$ be locally simplified bases for $C_1$ and $C_2$ respectively. Consider $C_1 \otimes C^*_2$. Let
\[ u_0 = \sum_{i=0}^{2k + 2\ell} x_i y^*_i \qquad \qquad u_1 = \sum_{\substack{i=1 \\ i \ \odd}}^{2k+2\ell+1} x_i y^*_{i-1}. \]
The elements
\begin{align*}
	& x_iy^*_i, \quad 0\leq i \leq 2k+2\ell  &&(0, \tau_x-\tau_y)  \\
	& x_iy^*_{i-1}, \quad i \textup{ odd}, 1\leq i \leq 2k-1 &&(1, \tau_x-\tau_y) \\
	& x_iy^*_{i-1}, \quad i \textup{ odd}, 2k+1 \leq i \leq 2k+2\ell+1 &&(n, \tau_x-\tau_y) \\
	& x_i y_{i+1}^*, \quad i \textup{ odd}, 1\leq i \leq 2k-1  &&(0, \tau_x-\tau_y+n) \\
	& x_i y_{i+1}^*, \quad i \textup{ odd}, 2k+1 \leq i \leq 2k+2\ell-1  &&(0, \tau_x-\tau_y+1)
\end{align*}
form a direct summand of $C\{S(n)\}$ where $\partial^\horz_n u_1 = u_0$ and $u_0$ is a generator of the vertical homology. The filtration level of $u_0$ is $(0, \tau_x-\tau_y)$ and of $u_1$ is $(n, \tau_x-\tau_y)$. Thus $a_1(C_1 \otimes C^*_2)=n$ and so $ [(1, n)^k, (n, 1)^k] - ((1, n)^k, (n, 1)^\ell, m, \ldots) =(n, \ldots)$.
\end{proof}

\begin{lemma}\label{lem:1nn1mb}
Let $\ell$, $m$, and $n$ be positive integers. 
Suppose $\ell < k$ and $m>n$. Then 
\[ ((1, n)^k, (n, 1)^\ell, -m, \ldots) = [(1, n)^k, (n, 1)^k] - (n, \ldots). \]
\end{lemma}

\begin{proof}
Let $C_1$ be a complex such that $\ba(C_1)= [(1, n)^k, (n, 1)^k]$ and $C_2$ a complex such that $\ba(C_2)=((1, n)^k, (n, 1)^\ell, -m, \ldots)$. Let $\{x_i\}$ and $\{y_i\}$ be locally simplified bases for $C_1$ and $C_2$ respectively. Consider $C_1 \otimes C_2$. Let
\[u_0 = \sum_{i=0}^{2k+2\ell} x_i y^*_i \qquad \qquad u_1 = \sum_{\substack{i=1 \\ i \ \odd}}^{2k+2\ell+1} x_iy^*_{i-1}. \]
The elements
\begin{align*}
	& x_iy^*_i, \quad 0\leq i \leq 2k+2\ell  &&(0, \tau_x-\tau_y)  \\
	& x_iy^*_{i-1}, \quad i \textup{ odd}, 1\leq i \leq 2k-1 &&(1, \tau_x-\tau_y) \\
	& x_iy^*_{i-1}, \quad i \textup{ odd}, 2k+1 \leq i \leq 2k+2\ell+1 &&(n, \tau_x-\tau_y) \\
	& x_i y_{i+1}^*, \quad i \textup{ odd}, 1\leq i \leq 2k-1  &&(0, \tau_x-\tau_y+n) \\
	& x_i y_{i+1}^*, \quad i \textup{ odd}, 2k+1 \leq i \leq 2k+2\ell-1  &&(0, \tau_x-\tau_y+1)
\end{align*}
form a direct summand of $C\{S(n)\}$ where $\partial^\horz u_1 = u_0$ and $u_0$ is a generator of the vertical homology. The filtration level of $u_0$ is $(0, \tau_x-\tau_y)$ and of $u_1$ is $(n, \tau_x-\tau_y)$. Thus $a_1(C_1 \otimes C^*_2)=n$ and so $ [(1, n)^k, (n, 1)^k] - ((1, n)^k, (n, 1)^\ell, -m, \ldots) =(n, \ldots)$.
\end{proof}

\noindent The next two lemmas concern case \ref{item:1nn1nm} from Lemma \ref{lem:forms}.

\begin{lemma}\label{lem:1nnnm}
Let $\ell$ be a non-negative integer, and let $m$ and $n$ be positive integers such that $m\geq 2$. 
Then 
\[ ((1, n)^k, (n, 1)^\ell, n, m, \ldots) = [(1, n)^k, (n, 1)^k] + (n, \ldots). \]
\end{lemma}

\begin{proof}
Let $C_1$ be a complex such that $\ba(C_1)= ((1, n)^k, (n, 1)^\ell, n, m, \ldots) $ and $C_2$ a complex such that $\ba(C_2)= [(1, n)^k, (n, 1)^k]$. Let $\{x_i\}$ and $\{y_i\}$ be locally simplified bases for $C_1$ and $C_2$ respectively. Consider $C_1 \otimes C_2$. Let
\[u_0 = \sum_{i=0}^{2k+2\ell+1} x_i y^*_i \qquad \qquad u_1 = \sum_{\substack{i=1 \\ i \ \odd}}^{2k+2\ell+1} x_iy^*_{i-1}. \]
The elements
\begin{align*}
	& x_iy^*_i, \quad 0\leq i \leq 2k+2\ell+1  &&(0, \tau_x-\tau_y)  \\
	& x_iy^*_{i-1}, \quad i \textup{ odd}, 1\leq i \leq 2k-1 &&(1, \tau_x-\tau_y) \\
	& x_iy^*_{i-1}, \quad i \textup{ odd}, 2k+1 \leq i \leq 2k+2\ell+1 &&(n, \tau_x-\tau_y) \\
	& x_i y_{i+1}^*, \quad i \textup{ odd}, 1\leq i \leq 2k-1  &&(0, \tau_x-\tau_y+n) \\
	& x_i y_{i+1}^*, \quad i \textup{ odd}, 2k+1 \leq i \leq 2k+2\ell+1  &&(0, \tau_x-\tau_y+1)
\end{align*}
form a direct summand of $C\{S(n)\}$. Consider also the element $x_{2k+2\ell+2}y^*_{2k+2\ell+2}$ in filtration level $(0, \tau_x-\tau_y+1-m)$. The chain $u_0+x_{2k+2\ell+2}y^*_{2k+2\ell+2}$ generates the vertical homology, and the projection of this chain to the subquotient complex $C\{S(n)\}$ is $u_0$. Moreover, $\partial^\horz_n u_1 = u_0$, and so  $a_1(C_1 \otimes C^*_2)=n$. 
\end{proof}

\begin{lemma}\label{lem:1nnnmb}
Let $\ell$ be a non-negative integer, and $m$ and $n$ be positive integers. Then 
\[ ((1, n)^k, (n, 1)^\ell, n, -m, \ldots) = [(1, n)^k, (n, 1)^k] + (n, \ldots). \]
\end{lemma}

\begin{proof}
Let $C_1$ be a complex such that $\ba(C_1)= ((1, n)^k, (n, 1)^\ell, n, -m, \ldots) $ and $C_2$ a complex such that $\ba(C_2)= [(1, n)^k, (n, 1)^k]$. Let $\{x_i\}$ and $\{y_i\}$ be locally simplified bases for $C_1$ and $C_2$ respectively. Consider $C_1 \otimes C_2$. Let
\[u_0 = \sum_{i=0}^{2k+2\ell+1} x_i y^*_i \qquad \qquad u_1 = \sum_{\substack{i=1 \\ i \ \odd}}^{2k+2\ell+1} x_iy^*_{i-1}. \]
The elements
\begin{align*}
	& x_iy^*_i, \quad 0\leq i \leq 2k+2\ell+1  &&(0, \tau_x-\tau_y)  \\
	& x_iy^*_{i-1}, \quad i \textup{ odd}, 1\leq i \leq 2k-1 &&(1, \tau_x-\tau_y) \\
	& x_iy^*_{i-1}, \quad i \textup{ odd}, 2k+1 \leq i \leq 2k+2\ell+1 &&(n, \tau_x-\tau_y) \\
	& x_i y_{i+1}^*, \quad i \textup{ odd}, 1\leq i \leq 2k-1  &&(0, \tau_x-\tau_y+n) \\
	& x_i y_{i+1}^*, \quad i \textup{ odd}, 2k+1 \leq i \leq 2k+2\ell+1  &&(0, \tau_x-\tau_y+1)
\end{align*}
form a direct summand of $C\{S(n)\}$. The chain $u_0$ generates the vertical homology, and $\partial^\horz u_1 = u_0$. Thus, $a_1(C_1\otimes C^*_2)=n$.
\end{proof}

\noindent Lastly, in the next two lemmas, we consider case \ref{item:1nn1c} from Lemma \ref{lem:forms}.

\begin{lemma}\label{lem:1nkn1k}
Let $m$ and $n$ be positive integers. 
Then 
\[ ((1, n)^k, (n, 1)^k, m, \ldots) = [(1, n)^k, (n, 1)^k] + (\textup{max}\{m, n\}, \ldots). \]
\end{lemma}

\begin{proof}
Let $C_1$ be a complex such that $\ba(C_1)=((1, n)^k, (n, 1)^k, m, \ldots)$ and $C_2$ such that $\ba(C_2)=[(1, n)^k, (n, 1)^k]$. Let $\{x_i\}$ and $\{y_i\}$ be locally simplified bases for $C_1$ and $C_2$ respectively. We consider $C_1 \otimes C^*_2$. Let
\[ u_0 = \sum_{i=0}^{4k} x_i y^*_i \qquad \qquad u_1 = \sum_{\substack{i=1 \\ i \ \odd}}^{4k+1} x_i y^*_{i-1}. \]
The elements
\begin{align*}
	& x_iy^*_i, \quad 0\leq i \leq 4k  &&(0, \tau_x-\tau_y)  \\
	& x_iy^*_{i-1}, \quad i \textup{ odd}, 1\leq i \leq 2k-1 &&(1, \tau_x-\tau_y) \\
	& x_iy^*_{i-1}, \quad i \textup{ odd}, 2k+1 \leq i \leq 4k-1 &&(n, \tau_x-\tau_y) \\
	& x_{4k+1}y^*_{4k},  &&(m, \tau_x-\tau_y) \\
	& x_i y_{i+1}^*, \quad i \textup{ odd}, 1\leq i \leq 2k-1  &&(0, \tau_x-\tau_y+n) \\
	& x_i y_{i+1}^*, \quad i \textup{ odd}, 2k+1 \leq i \leq 4k-1  &&(0, \tau_x-\tau_y+1)
\end{align*}
form a direct summand of $C\{S(n)\}$. 
The chain $u_0$ generates the vertical homology and $\partial^\horz u_1=u_0$. The filtration level of $u_0$ is $(0, \tau_x-\tau_y)$ and that of $u_1$ is $(\max\{m, n\}, \tau_x-\tau_y)$, and so $a_1(C_1 \otimes C^*_2)= \max\{m, n\}$.
\end{proof}

\begin{lemma}
\label{lem:last}
Let $m$ and $n$ be positive integers. 
Then 
\[ ((1, n)^k, (n, 1)^k, -m, \ldots) = [(1, n)^k, (n, 1)^k] - (\textup{max}\{m, n\}, \ldots). \]
\end{lemma}

\begin{proof}
Let $C_1$ be a complex such that $\ba(C_1)=[(1, n)^k, (n, 1)^k]$ and $C_2$ a complex such that $\ba(C_2)=((1, n)^k, (n, 1)^k, -m, \ldots)$. Let $\{x_i\}$ and $\{y_i\}$ be locally simplified bases for $C_1$ and $C_2$ respectively. We consider $C_1 \otimes C^*_2$. Let
\[ u_0 = \sum_{i=0}^{4k} x_i y^*_i \qquad \qquad u_1 = x_{4k} y^*_{4k+1} + \sum_{\substack{i=1 \\ i \ \odd}}^{4k-1} x_i y^*_{i-1}. \]
The elements
\begin{align*}
	& x_iy^*_i, \quad 0\leq i \leq 4k  &&(0, \tau_x-\tau_y)  \\
	& x_iy^*_{i-1}, \quad i \textup{ odd}, 1\leq i \leq 2k-1 &&(1, \tau_x-\tau_y) \\
	& x_iy^*_{i-1}, \quad i \textup{ odd}, 2k+1 \leq i \leq 4k-1 &&(n, \tau_x-\tau_y) \\
	&x_{4k} y^*_{4k+1},  &&(m, \tau_x-\tau_y) \\
	& x_i y_{i+1}^*, \quad i \textup{ odd}, 1\leq i \leq 2k-1  &&(0, \tau_x-\tau_y+n) \\
	& x_i y_{i+1}^*, \quad i \textup{ odd}, 2k+1 \leq i \leq 4k-1  &&(0, \tau_x-\tau_y+1)
\end{align*}
form a direct summand of $C\{S(n)\}$, with $\tau_x=\tau(C_1)$ and $\tau_y=\tau(C_2)$. The chain $u_0$ generates the vertical homology and $\partial^\horz u_1=u_0$. The filtration level of $u_0$ is $(0, \tau_x-\tau_y)$ and that of $u_1$ is $(\max\{m, n\}, \tau_x-\tau_y)$. Thus, $a_1(C_1 \otimes C^*_2) = \max \{m, n\}$.
\end{proof}

This completes the analysis of the cases in Lemma \ref{lem:forms}. We are now ready to prove Proposition \ref{prop:ArchA}, i.e., we will show that the Archimedean equivalence class of $C_n=[1, n, n, 1]$ satisfies Property A.

\begin{proof}[Proof of Proposition \ref{prop:ArchA}]
Consider a complex $C$. We will show that if $C$ is Archimedean equivalent to $C_n$, then $C = k \cdot C_n + C'$ for some $|C'| \ll C_n$.

If $C$ is a multiple of $C_n$, then we are done. Otherwise, $\ba(C)$ is of one of the forms in Lemma \ref{lem:forms}. But by Lemma \ref{lem:agg} and Lemmas \ref{lem:m>n}-\ref{lem:last}, any such complex either dominates $C_n$ or is equal to $k \cdot C_n + C'$, where $|C'| \ll C_n$.
\end{proof}
 
 
\section{Realizing complexes by knots}
\label{sec:knots}
In this section, we give a family of topologically slice knots that realize the Archimedean equivalence class of $C_n = [1, n, n, 1]$. The knots we will use are
\[ K_n = D_{n, n+1} \# -T_{n, n+1}, \]
where $D$ denotes the (positively clasped, untwisted) Whitehead double of the right-handed trefoil, $D_{n, n+1}$ denotes the $(n, n+1)$-cable of $D$ (where $n$ denotes the longitudinal winding), and $T_{n, n+1}$ denotes the $(n, n+1)$-torus knot.

These knots were used in \cite{Homsmooth} to give an infinite rank \emph{subgroup} of $\cC_{TS}$, and Lemma 6.10 of that paper shows that 
\[ \ba(K_n) = (1, n, \ldots).\]
However, this is not sufficient to conclude that $K_n$ represents the same Archimedean equivalence class as $[1, n, n, 1]$; see, for example, Lemma \ref{lem:m>n}.

\begin{proposition}
\label{prop:DArch}
	Let $K_n = D_{n, n+1} \# -T_{n, n+1}$ where $n>1$. Then
	\[ [\CFKi(K_n)] \sim_A [1, n, n, 1]. \]
\end{proposition}

\noindent We first prove two lemmas partially determining $\ba(D_{n, n+1})$ and $\ba(T_{n, n+1})$.

\begin{lemma} 
\label{lem:D}
	For the knot $D_{n, n+1}$ where $n>2$, we have
	\[ \ba( D_{n, n+1}) = (1, n, 1, n-2, \ldots).\]
	In particular, $[\CFKi(D_{n, n+1})] \sim_A [1, n, n, 1]$.
\end{lemma}

\begin{proof}
By Lemma 6.12 of \cite{Homsmooth}, the knot $D$ is $\varep$-equivalent to $T_{2, 3}$. Thus, by \cite[Proposition 4]{Homsmooth}, we may consider $\CFKi(T_{2,3; n, n+1})$ instead of $\CFKi(D_{n, n+1})$, where $T_{2, 3; n, n+1}$ denotes the $(n, n+1)$-cable of $T_{2,3}$. The advantage of this approach is that $T_{2,3; n, n+1}$ is an $L$-space knot \cite[Theorem 1.10]{HeddencablingII} (cf. \cite{HomLspace}), and so its knot Floer complex is completely determined by its Alexander polynomial \cite[Theorem 1.2]{OSlens}.

By \cite[Lemma 6.7]{Homsmooth} (see also the proof of \cite[Proposition 6.1]{HeddenLivRub}), we have that
\[ \Delta_{T_{n, n+1}}(t) = \sum_{i=0}^{n-1} t^{ni} - t \sum_{i=0}^{n-2} t^{(n+1)i}.\]
In particular, if $n \geq 3$, then 
\[ \Delta_{T_{n,n+1}}(t) = t^{n^2-n} - t^{n^2-n-1} + t^{n^2-2n} - t^{n^2-2n-2} + t^{n^2-3n} + \textup{lower order terms}. \]
The Alexander polynomial of $T_{2, 3; n, n+1}$ is
\begin{align*}
	\Delta_{T_{2, 3; n, n+1}}(t) &= (t^{2n}-t^n +1) ( t^{n^2-n} - t^{n^2-n-1} + t^{n^2-2n} - t^{n^2-2n-2} + t^{n^2-3n} + \textup{lower order terms}) \\ 
	&= t^{n^2+n} - t^{n^2+n-1} + t^{n^2-1} - t^{n^2-2} + t^{n^2-n} + \textup{lower order terms}.
\end{align*}
By \cite[Lemma 6.5]{Homsmooth}, the exponents (or, more precisely, their differences) of the Alexander polynomial of an $L$-space knot determine the invariants $a_i$. Thus,
\[ \ba(T_{2, 3; n, n+1}) = (1, n, 1, n-2, \dots). \]
By Lemma \ref{lem:mn}, 
\[  (1, n, 1, n-2, \dots) = [1, n, n, 1] + (1, n-2, \ldots) \]
and by Lemma \ref{lem:agg}
\[ (1, n-2, \ldots) \ll (1, n, \ldots). \]
Therefore, $[\CFKi(D_{n, n+1})] \sim_A [1, n, n, 1]$, as desired.
\end{proof}

\begin{lemma}
For the torus knot $T_{n, n+1}$ where $n$ is a positive integer,
\[ \ba( T_{n, n+1}) = (1, n-1, 2, \ldots).\]
In particular, $[\CFKi(T_{n, n+1})] \sim_A [1, n-1, n-1, 1]$.	
\end{lemma}

\begin{proof}
It follows from the proof of Lemma \ref{lem:D} that
\[ \Delta_{T_{n,n+1}}(t) = t^{n^2-n} - t^{n^2-n-1} + t^{n^2-2n} - t^{n^2-2n-2} + \textup{lower order terms}. \]
Hence by \cite[Lemma 6.5]{Homsmooth},
\[ \ba(T_{n,n+1}) = (1, n-1, 2, \ldots), \]
and by Lemma \ref{lem:fortorusknots}
\[ (1, n-1, 2, \ldots) = [1, n-1, n-1, 1] + (2, \ldots), \]
implying that
\[ [\CFKi(T_{n, n+1})] \sim_A [1, n-1, n-1, 1] \]
since $(2, \ldots) \ll [1, n-1, n-1, 1]$ by  Lemma \ref{lem:agg}. 
\end{proof}

With these lemmas in place, we are ready to show that $[\CFKi(K_n)] \sim_A [1, n, n, 1]$.

\begin{proof}[Proof of Proposition \ref{prop:DArch}]
By the preceding two lemmas, when $n>2$,
\[ [\CFKi(D_{n, n+1})] \sim_A [1, n, n, 1]  \qquad \textup{and} \qquad [\CFKi(T_{n, n+1})] \sim_A [1, n-1, n-1, 1]. \]
Moreover,  Lemma \ref{lem:agg} implies that
\[ [1, n, n, 1] \gg [1, n-1, n-1, 1]. \]
Hence these knots do represent the desired Archimedean equivalence class, namely,
\[ [\CFKi(D_{n, n+1} \# -T_{n, n+1})] \sim_A [1, n, n, 1].\] 
When $n=2$, it is straightforward to check that $ \Delta_{T_{2, 3; 2, 3}} (t) = t^6 - t^5 + t^3 -t +1 $
and so $\ba(T_{2,3;2,3}) = [1, 2, 2, 1]$.
\end{proof}

\section{Proof of Theorem \ref{thm:main}}
\label{sec:proof}
We end with the proof that $\cC_{TS}$ has a summand isomorphic to $\Z^\infty$.

\begin{proof}[Proof of Theorem \ref{thm:main}]
By Proposition \ref{prop:DArch}, for any integer $n \geq 2$, the knot
\[ K_n = D_{n, n+1} \# -T_{n, n+1} \]
is in the Archimedean equivalence class $[1, n, n, 1]$, and by Lemma \ref{lem:agg}
\[ [\CFKi(K_{n+1}) \gg [\CFKi(K_n)]. \]
By Proposition \ref{prop:ArchA}, the class $[1, n, n, 1]$ satisfies Property A. Therefore, by Proposition \ref{prop:propertyA}, the elements $[\CFKi(K_n)] \in \cCFK$ determine a collection of linearly independent surjective homomorphisms 
\[ \cCFK \rightarrow \Z. \]
The knot $D$ has Alexander polynomial one, and hence is topologically slice. Thus, the cable $D_{n, n+1}$ is topologically concordant to the underlying pattern torus knot $T_{n, n+1}$, and so $D_{n, n+1} \# -T_{n, n+1}$ is topologically slice.
\end{proof}

\begin{proof}[Proof of Corollary \ref{cor:main}]
The corollary follows from the observation that the knots $K_n$ are topologically slice, and thus contained in the kernel $\cA$ of Levine's homomorphism. In particular, we have an infinite collection of surjective linearly independent homomorphisms $\cA \rightarrow \Z$.
\end{proof}

\bibliographystyle{amsalpha}

\bibliography{mybib}

\end{document}


\begin{abstract}
We describe an error in the proof of a key proposition of \cite{Homsummand}, which was necessary for the proof of the main result. Alternate proofs of the main result are given by Ozsv\'ath-Stipsicz-Szab\'o and Dai-Hom-Stoffregen-Truong.
\end{abstract}

\maketitle

In \cite{Homsummand},  we constructed a homomorphism
\[ \cC \to \cCFK \]
where $\cC$ denotes the smooth knot concordance group and $\cCFK$ consists of knots Floer complexes modulo $\varep$-equivalence. We then applied the Hahn Embedding Theorem to show that
\[ \cCFK \hookrightarrow \R^X. \]
These results still stand.

However, Proposition 1.3 of \cite{Homsummand}, which claimed that this embedding could be used to construct an infinite family of $\Z$-valued homomorphisms from $\cCFK$, is false. The error is that we quoted a version of the Hahn Embedding Theorem for vector spaces \cite{HausnerWendel}, not groups. In particular, condition (2) above Definition 1.2 of \cite{Homsummand} does not hold for groups.

To see that Proposition 1.3 is false for groups, consider the following example due to B. Gordon and found in \cite{Clifford}. Order $\Q^2$ lexicographically. Let $G$ be the (ordered) subgroup of $\Q^2$ generated by $(p_n^{-1}, np_n^{-1})$ where $p_n$ denotes the $n^{\text{th}}$ prime. Then $(0,1) \in G$ satisfies the hypothesis of Proposition 1.3, but not the conclusion. In particular, $(0,1) \in G$ satisfies Property A but the image of projection from $G$ onto its second coordinate is not isomorphic to $\Z$.

The main purpose of defining $\Z$-valued concordance homomorphisms was to prove the following result about $\CTS$, the subgroup of the smooth concordance group generated by topologically slice knots.

\begin{Theorem}\label{thm1}
The group $\CTS$ contains a direct summand isomorphic to $\Z^\infty$.
\end{Theorem}

\noindent An independent proof of Theorem \ref{thm1} is given by Ozsv\'ath, Stipsicz, and Szab\'o \cite{OSS}, using their concordance homomorphism $\Upsilon$.\footnote{Ironically, Ozsv\'ath, Stipsicz, and Szab\'o state in their introduction that $\Upsilon$ was inspired by \cite{Homsummand}.}

Recent work of Dai, Stoffregen, Truong, and the author \cite{DHSTCFK} provides a new infinite family of $\Z$-valued concordance homomorphisms $\varphi_n$ that factor through $\cCFK$. The definition of $\varphi_n$ still relies on the total order of $\cCFK$, although it does not rely on the Hahn Embedding Theorem. The homomorphisms $\varphi_n$ give an alternate proof of Theorem \ref{thm1}, using the same knots as in \cite{Homsummand}.

\bibliographystyle{amsalpha}
\bibliography{bib2}